\documentclass[12pt,a4paper]{amsart}

\usepackage{graphicx}
\usepackage[small]{caption}
\usepackage{datetime,verbatim,array}
\usepackage{amsthm, amsmath, amsfonts, amssymb}
\usepackage{mathrsfs, rotating}
\usepackage[usenames]{color}
\usepackage[margin=2.9cm]{geometry}
\usepackage[latin2]{inputenc}
\usepackage{t1enc}
\usepackage{epstopdf} 
\usepackage[colorlinks, citecolor=blue, filecolor=black, linkcolor=red, urlcolor=blue]{hyperref}
 
\usepackage{comment}
\usepackage[all,cmtip]{xy}

\usepackage{pinlabel}

\theoremstyle{plain}
\newtheorem{theorem}{Theorem}[section]

\newtheorem{lemma}[theorem]{Lemma}
\newtheorem{proposition}[theorem]{Proposition}
\newtheorem{corollary}[theorem]{Corollary}

\newtheorem{remark}[theorem]{Remark}

\def\N{\mathbb{N}}				

\def\pminus{+-}
\def\mplus{-+}
\def\mminus{--}

\def\a{\alpha}
\def\b{\beta}
\def\g{\gamma}

\newcommand{\nc}{\newcommand}

\nc{\imin}{i_{\textrm{min}}}

\nc{\p}[1]{\medskip\noindent{\em #1.}}
\nc{\margin}[1]{\marginpar{\scriptsize #1}}

\title[Super efficient geodesics in the complex of curves]{Super efficient geodesics in the complex of curves}

\author{Xifeng Jin and William Menasco}

\address{Xifeng Jin\\
Department of Mathematics\\
University at Buffalo--SUNY\\
Buffalo, NY 14260-2900, USA\\
xifengji@buffalo.edu}

\address{William W. Menasco\\
Department of Mathematics\\
University at Buffalo--SUNY\\
Buffalo, NY 14260-2900, USA\\ menasco@buffalo.edu}



\date{\currenttime \ \ \today}

\begin{document}
\vspace*{-1.25in}

\begin{abstract}
We introduce a subclass of the efficient geodesics, called {\em super efficient geodesics}, that the bound of intersection number in the initial efficiency only depends on the genus of  surface.  For any two vertices, $v , w \in \mathcal{C}(S_g)$, in the complex of curves 
of a closed oriented surface of genus $g \geq 2 $, and any efficient geodesic, $v = v_1 , \cdots , v_{\text{d}}=w$, it was previously established 
by Birman, Margalit and the second author \cite{BMM} that
there is an explicitly computable list of at most ${\text{d}}^{(6g-6)}$ candidates for the $v_1$ vertex.  In this note we establish a bound for this computable list for the super efficient geodesics that is independent of ${\text{d}}$-distance and only dependent on genus.  The proof relies on an intersection growth inequality (IGI) between the intersection number of two curves and their distance in the complex of curves, together with a thorough analysis of the dot graph associated with the intersection sequence. 
\end{abstract}

\maketitle

\section{Introduction}
\label{section: introduction}

\subsection{Initial background}
Let $S_g$ be a closed oriented surface of genus $g \geq 2$.  The \emph{complex of curves} or \emph{curve complex} of the surface $S_g$ is a simplicial complex such that each vertex corresponds to the isotopy class of an essential simple closed curve, and $n+1$ vertices form an $n$-simplex of $\mathcal{C}(S_g)$ if their representatives can be realized disjointly.  At times we will use the words ``vertices'' and ``curves'' interchangeably, and use \emph{curve} as essential simple closed curve.  For any two curves $v$ and $w$ in $\mathcal{C}(S_g)$, the \emph{distance} $d(v,w)$ is the minimal number of edges in $\mathcal{C}(S_g)$ connecting $v$ to $w$.  A \emph{geodesic} in $\mathcal{C}(S_g)$ is a sequence of vertices $\Gamma=(\g_i)_{i \in I}$ such that $d (\g_i,\g_j)=|i-j|$ for all $i,j \in I$. 

Harvey \cite{Harvey} introduced the curve complex in 1978 as a tool for studying the mapping class group of surfaces.  In 1996, Masur and Minsky \cite{MM1} proved the seminal result that $\mathcal{C}(S_g)$ is $\delta$-hyperbolic. However, the computation of distance in $\mathcal{C}(S_g)$ is still very intimidating. The first attempt of such algorithm was due to Leasure \cite{L} in his thesis in 2002. After that, there are several other algorithms found by Shackleton \cite{S},  Watanabe \cite{Watanabe1} and Webb \cite{Webb}, in which all algorithms utilize the notion of \emph{tight geodesics} that were first constructed in a sequel \cite{MM2} of Masur-Minky's work \cite{MM1}. In \cite{BMM} Birman-Margalit-Menasco introduce \emph{efficient geodesics}, an alternative finite set of geodesics for computing the distance in $\mathcal{C}(S)$.  
 
To elaborate, suppose $v_0, v_1, \cdots, v_{\text{d}}$ is a geodesic of length ${\text{d}} \geq 3$ in $\mathcal{C}(S_g)$, and $\a_0,\a_1,\cdots , \a_{\text{d}}$ are associated curve representatives in $S_g$.  Initially focussing on the triple, $\a_0, \a_1, \a_{\text{d}}$, assume that they pairwise are in minimal position---$\a_0 \cap \a_1 = \varnothing$, and no subarc of $\a_0$ or $\a_1$ is cobounding a bigon disc with a subarc of $\a_{\text{d}}$.   A \emph{reference arc} for $\a_0, \a_1, \a_{\text{d}}$ is an arc, $\g \subset S_g \setminus (\a_0 \cup \a_{\text{d}})$, that is in minimal position with $\a_1$---again, no cobounding a bigon disc between $\a_1$ and $\g$. The geodesic $v_0, v_1, \cdots, v_{\text{d}}$ is \emph{initially efficient} if $|\a_1 \cap \g| \leq {\text{d}} - 1$ for all possible reference arcs. Moreover, the geodesic is called \emph{efficient}, if the oriented geodesic $v_k, v_{k+1}, \cdots, v_{\text{d}}$ is initially efficient for $0 \leq k \leq {\text{d}} - 3$ and the oriented geodesic $v_{\text{d}}, v_{{\text{d}}-1}, v_{{\text{d}}-2}, v_{{\text{d}}-3}$ is also initially efficient.

The existence of initially efficient geodesics is established in \cite{BMM}.
Specifically, it is shown that  there that an explicitly computable list of at most $$  {\text{d}}^{\ 6g - 6} $$ vertices, $v_1$, which can appear as the first vertex of an (initially) efficient geodesic $v = v_0,v_1, \cdots,v_{\text{d}} = w$. In \cite{BMM} it is argued that any geodesic could be replaced by an initially efficient geodesic; that is, $|\a_1 \cap \g| \leq \text{d} - 1$, for all possible reference arcs $\g$.  Here, the notation $|\a_1 \cap \g|$ refers to the number of minimal intersections which is achieved when there is no cobounding bigon disc.  Equivalently, when $\a_1$ and $\g$ intersect minimally within the isotopy class $v_1$.

\subsection{Main results}
\label{main results}
Our key advancement in this note is the establishment of a subclass of initially efficient geodesics, called {\em initially super efficient geodesics}, for which the intersection bound with the reference arcs is $\min({\text{d}} - 1, 15 \cdot (6g - 8))$ for genus $g \geq 3$ and $\min({\text{d}} - 1, 44)$ for $g=2$.  
The {\em super efficient geodesics} follow similarly as the efficient geodesics in \cite{BMM} once we establish the initially super efficient geodesics. In particular, the existence of initially super efficient geodesics comes from changing the distance-dependent intersection bound to one that is dominated by a linear function of genus. 

\begin{theorem}
\label{theorem: super efficiency}
Let $v$ and $w$ be two vertices in $\mathcal{C}(S_{g \geq 2})$ with distance $d(v,w) \geq 3$. There exists an (initially) super efficient geodesic $v = v_0,v_1, \cdots ,v_\text{d} = w$  in $\mathcal{C}(S_g)$ connecting $v$ and $w$. More precisely, suppose $\a_0$, $\a_1$ and $\a_{\text{d}}$ are minimally intersecting representatives of
$v_0$, $v_1$ and $v_{\text{d}}$, respectively.  Let $\g$ be any reference arc for the pair $(\a_0 , \a_{\text{d}})$.
Then for $g=2$, we have $| \a_1 \cap \g | \leq \min({\text{d}} - 1, 44) $; and for $g \geq 3$, we have $| \a_1 \cap \g | \leq \min({\text{d}} - 1, 15 \cdot (6g - 8)) $.
\end{theorem} 

For a pair of curves, $(\a, \b)$, associated with a pair of vertices, $(v,w)$, we again assume $\a$ and $\b$ are in minimal position.  Define $ i(v,w) :=|\a \cap \b|$.  The {\em{complexity}} of a path, $v=v_0, v_1, \cdots, v_{\text{d}}=w$ in $\mathcal{C}(S_g)$ is then defined to be the number, $\sum_{k=1}^{{\text{d}}-1} (i(v_0, v_k) + i(v_k, v_{\text{d}}))$.
Then, among the geodesics between two vertices, we have the following inclusion relation,
$$\{ \text{geodesics of minimal complexity} \} \subset \{ \text{ISEG} \} \subset \{ \text{IEG}\}, $$
where the ISEG stands for initially super efficient geodesics and the IEG stands for initially efficient geodesics.

Now building on Theorem \ref{theorem: super efficiency} we obtain a new bound on the size of this explicit list of the candidates for $v_1$ vertex in the initially super efficient geodesics that is independent of distance and only dependent on genus.  In particular, we have the following corollary.

\begin{corollary}[Super efficiency]
\label{corollary: candidates for v1}
If $v = v_0,v_1, \cdots,v_\text{d} = w$ is an (initially) super efficient geodesic of $\mathcal{C}(S_g)$ with $d(v,w) = {\text{d}} \geq 3$, then, for $g>2$, there is an explicitly computable list of at most 
$$ {[15 \cdot (6g - 8) + 1]}^{6g -6}$$
vertices $v_1$ that can appear as the first vertex on an (initially) super efficient geodesic.  When $g=2$ the bound on this list is $ 45^6$.
\end{corollary}

\subsection{Test for a geodesic path}
The new idea advancing our understanding of efficient geodesics is the application of a simple necessary condition for when a given vertex sequence is a geodesic.
Specially, for a pair of natural numbers, $(\text{d} , g) \in \N^2$ with $\text{d} \geq 3$ and $g \geq 2$, define the {\em minimal intersection function}, $\mathcal{I}(\text{d},g) := \min \{| \a \cup \b | : \a, \b \subset S_g \ {\rm and} \ d(\a ,\b) = \text{d} \} $.  Currently, few exact values of $\mathcal{I}$ are known.  We do know the $\mathcal{I}(3,2) =4$ and for $g \geq 3$ we know that $\mathcal{I}(3,g) = 2g-1$ \cite{AH}.  (Recall, pairs of curves at distance $3$ or greater are called {\em filling pairs}.)  Additional, we know that $\mathcal{I}(2,4) = 12$ \cite{GMMM}.  

Now, let $A = (\a_i)_{i \in I}$ we a sequence of essential simple curves in $S_g$ representing the vertices of $\Gamma$ such that $|\a_i \cap \a_{i+1}| = 0$ for all $i , i+1 \in I$.  If $\Gamma$ is a geodesic, by definition $d( \g_i \g_j) = |i -j|$.  But, this implies that $|\a_i \cap \a_j| \geq \mathcal{I}(|i-j|,g)$.  This observation gives us a {\em necessary test} for $\Gamma$ (and $A$) being a geodesic sequence.  Additionally, we can expand on this necessary test by utilizing the triangle inequality.  Specially, for any curve $\delta \subset S_g$ we cannot have
\begin{equation}
\label{equation: delta}
(|\a_i \cap \a_j | \leq ) \ |\a_i \cap \delta | + | \a_j \cap \delta | < \mathcal{I}(|i-j| , g).
\end{equation} 

The argument in \cite{BMM} for the existence of efficient geodesics comes from the application of ``path surgery'' operations which replaces a geodesic $\Gamma$ with a new geodesic path, $\Gamma^\prime$, while strictly decreasing the intersections of representative curves with reference arcs.  In the proof the fact that the new sequence, $\Gamma^\prime$, is a geodesic only comes from it being the same length as $\Gamma$.  Our strategy for proving Theorem \ref{theorem: super efficiency} is to show that if the intersection of $\a_1$ with a reference arc is too large then, in order that there not be a curve $\delta$ satisfying equation \ref{equation: delta}, our geodesic path is of infinite length.
To implement this strategy we need a tight lower bound on the growth of $\mathcal{I}(\text{d}, g)$ as a function of $\text{d}$.  Our bound will be given by  the {\em intersection growth inequality} or {\em IGI}.

\begin{theorem}[IGI] \label{theorem: intersection growth}
Let two curves $\a, \b \subset S_g $ represent two vertices $v , w \in \mathcal{C} (S_{g })$ that realize their intersection number.
If $g=2$ and $d(v , w) = {\text{d}} \geq 4$ then
$$2^{{\text{d}} -4} \cdot 12 \leq |\a \cap \b|.$$  In particular, $2^{{\text{d}} -4} \cdot 12 \leq \mathcal{I}(\text{d} , 2)$.
If $g>2$ and $d(v , w) = {\text{d}} \geq 3$ then $$2^{{\text{d}} -3} \cdot (2g -1) \leq |\a \cap \b|,$$
that is,
$$d \leq 2 + \log_2\left(\frac{|\a \cap \b|}{g - 0.5}\right).$$
In particular, $2^{{\text{d}} -3} \cdot (2g -1) \leq \mathcal{I}(\text{d} , g)$ for $g \geq 3$.
\end{theorem}

\subsection{An aside---other lower bounds for $\mathcal{I}(\text{d},g)$}
\label{subsection: estimate}
The IGI takes into account of the genus of surface improving the following estimate given by Hempel \cite{Hempel}:
\begin{equation}
\label{equation: Hempel inequaltiy}
    d(\alpha, \beta) \leq 2 + 2 \cdot   \log_2(|\a \cap \b|)
\end{equation} 

Another estimate utilizes the topology of a surface was proved by Bowditch (Corollary 2.2 in \cite{BB}) and reformulated by Aougab, Patel and Taylor \cite{APT}. 

\begin{theorem}
For any two curves $\a$ and $\b$ on an oriented surface $S_{g,p}$ with Euler characteristic $|\chi(S_{g,p})| \geq 5$ and $|\a \cap \b| > 0$, 
\begin{equation}
\label{equation: Bowditchv inequality}
    d(\alpha, \beta) < 2 + 2 \cdot  \frac{ \log_2(|\a \cap \b|/2)}{\log_2((|\chi{(S_{g,p})|-2})/2)}.
\end{equation}
\end{theorem}

 From the asymptotic perspective, a much stronger result was provided by Aougab (Theorem 1.2 in \cite{Aougab}) to show the uniform hyperbolicity of curve graphs. 

\begin{theorem}
\label{theorem: Aougab inequality}
For each $\lambda \in (0,1)$, there is some $N = N(\lambda) \in \mathbb{N}$ such that if $\a, \b \in \mathcal{C}_0(S_{g,p})$, whenever $\xi(S_{g,p}) > N$ and $d_{\mathcal{C}}(\a, \b) \geq k$,
\begin{equation}
    \label{equation: Aougab inequality}
|\alpha \cap \beta| \geq \left(\frac{\xi(S_{g,p})^\lambda}{f(\xi(S_{g,p}))}\right)^{k-2},
\end{equation}
where $f(\xi) = O(\log_2(\xi))$.
\end{theorem}

For a closed oriented surface $S_{g \geq 4}$, Bowditch's estimate becomes 
$$d(\a, \b) < 2 + 2 \cdot  \frac{ \log_2(|\a \cap \b|/2)}{\log_2(g-2)}.$$ 
If $2 \leq g \leq 6$, then the upper bound in the
IGI is smaller. Otherwise, Bowditch's estimate is smaller as long as the intersection number of $\a \cap \b$ is sufficiently large. The bound in Theorem \ref{theorem: super efficiency} is obtained with the IGI. 

\subsection{Improving distance algorithm}
\label{algorithm}
We extend the discussion begun in the introduction of \cite{BMM} comparing the bounds on the size of the set of $v_1$ candidates
given by the use of super efficient geodesics and the bounds by the use of tight geodesics.  From that discussion we have the bound results
coming from the work of Webb \cite{Webb} which gives a bound of $2^{(72g+12) {\text min}\{{\text{d}} - 2,21\} }(2^{6g - 6} -1)$.  In particular, when $ {\text{d}} -2 \geq 21 $ and $g = 2$
we get a bound that is $ \sim~10^{75}$.  Whereas, a super efficiency bound for $g=2$ is $\sim~10^{10}$ and independent of distance.  Although an improvement, still such large bounds do not give a suitable understanding in how super efficient geodesics might be utilized in the implementation of an algorithm for computing distance. 

In \cite{BellWebb} Bell and Webb describe such a polynomial time algorithm based on the Masur-Minsky's tight geodesics technology for computing distance in the curve complex.  A key feature of their description is the ``train-track-type'' behavior of curves representing the vertices of a tight geodesic.  
A distance algorithm utilizing super efficient geodesics will exhibit analogous behavior in terms of ``parallel arcs'', but in order to more fully expand the discussion in this direction it is helpful to have the reader understand the details of the proof of Theorem \ref{theorem: super efficiency}.  We thus postpone this discussion until \S \ref{subsection: algorithm}.

\subsection{Outline of paper}
The paper is organized as follows. In \S \ref{section: Intersection}, we utilize the intersection number of a minimal filling pair as a lower bound and linear integer programming to establish the intersection growth inequality (IGI). 
In \S \ref{section: One-vertex triangulation and parallel arcs}, we use the one-vertex triangulation of surface to show that some parallel arcs occur as the intersection number is sufficiently large. \S \ref{section: efficient geodesics} reviews the techniques of efficient geodesics and curve surgery from \cite{BMM}.  In \S \ref{section: Obstruction to surgeries}, we describe a ``triangular shape'' dot graph will occur in the worst scenario and its corresponding parallel arcs on the surface. 
\S \ref{section: rainbow} defines these parallel arcs as a \emph{rainbow} and expands the dot graph machinery of \cite{BMM} by investigating the pattern hided in the dot graph that arises as a rainbow, which allows us to calculate the minimal intersection number that will create another rainbow for other curves. We complete the proof of Theorem \ref{theorem: super efficiency} by applying the IGI to the pattern of dot graph in the last \S \ref{section: super efficiency}. Then the super efficiency of Corollary \ref{corollary: candidates for v1} follows immediately.  Finally, we revisit the discussion of how our results could be utilized in the implementation of a distance computing algorithm in \S \ref{subsection: algorithm}.

\section*{Acknowledgement}
The second author wishes to thank Dan Margalit for numerous conversations he had with him during the development of this project.  Similarly, the second author thanks Joan Birman for numerous discussions on the curve complex and efficient geodesics.  We also wish to thank Kasra Rafi for alerting us to Brian Bowditch's growth rate result in \cite{BB}, and thank Tarik Aougab for letting us know his asymptotic growth rate in Theorem \ref{theorem: Aougab inequality} in \cite{Aougab}.  Finally, we would like to thank an anonymous reviewer for pointing out an error in our original statement of Theorem \ref{theorem: super efficiency} and many other helpful comments that improved the exposition.

\section{Establishing the IGI}
\label{section: Intersection}

The idea of our proof for the IGI is to utilize the {\em linear integer program} (LIP) approach that was developed in the work
of Glenn, Morrel, Morris and the second author \cite{GMMM} to establish the following result.

\begin{theorem} \label{theorem: DWH-G2-D4-I12}
The minimal intersection number for a filling pair, $\a, \b \subset S_2$, representing
vertices $v, w \in \mathcal{C}(S_2)$, respectively, with $d(v,w)=4$ is $12$.
\end{theorem}

\begin{proof}[Proof of Theorem \ref{theorem: intersection growth}] 

First, let us recall the LIP argument. Suppose that $(\a, \b)$ is a filling pair for $S_2$, and we split $S_2$ along $\a$ and consider the resulting properly embedded arcs of $\b$
in a genus one surface having two boundary components, $S_{1,2}$.  (As in \cite{BMM}, we can reduce to the case where $\a$ in non-separating so that the surface is connected.)
Assume that $\a$ and $\b$ are arranged to being minimally intersecting up to isotopy, we can assume that all the discs of $S_{1,2} \setminus \b$ are
$2k$-gons, $k \geq 2$.  Now consider an essential curve, $ c \subset S_{1,2}$, that intersect each arc of $\b$ at most once.  We notice that
$|c \cap \b| $ has to be sufficient so that $d(c, \b) \geq d (\a, \b) - 1$.  In particular, if $d(\a, \b) = 4$, we have to have
$ |c \cap \b| \geq 4$, since $4$ is the minimal intersection number for the pair $(c , \b)$ to be filling.

\begin{figure}[htbp]
\labellist
\small\hair 2pt
\pinlabel $w_6$ at 116 100
\pinlabel $w_2$ at 107 245
\pinlabel $w_4$ at 340 123
\pinlabel $w_1$ at 73 305
\pinlabel $w_3$ at 380 305
\pinlabel $w_5$ at 225 90
\pinlabel $w_6$ at 625 93
\pinlabel $w_2$ at 619 245
\pinlabel $w_4$ at 840 123
\pinlabel $w_1$ at 585 305
\pinlabel $w_3$ at 885 305
\pinlabel $w_5$ at 735 94
\endlabellist
\centering
\includegraphics[width=0.90\textwidth]{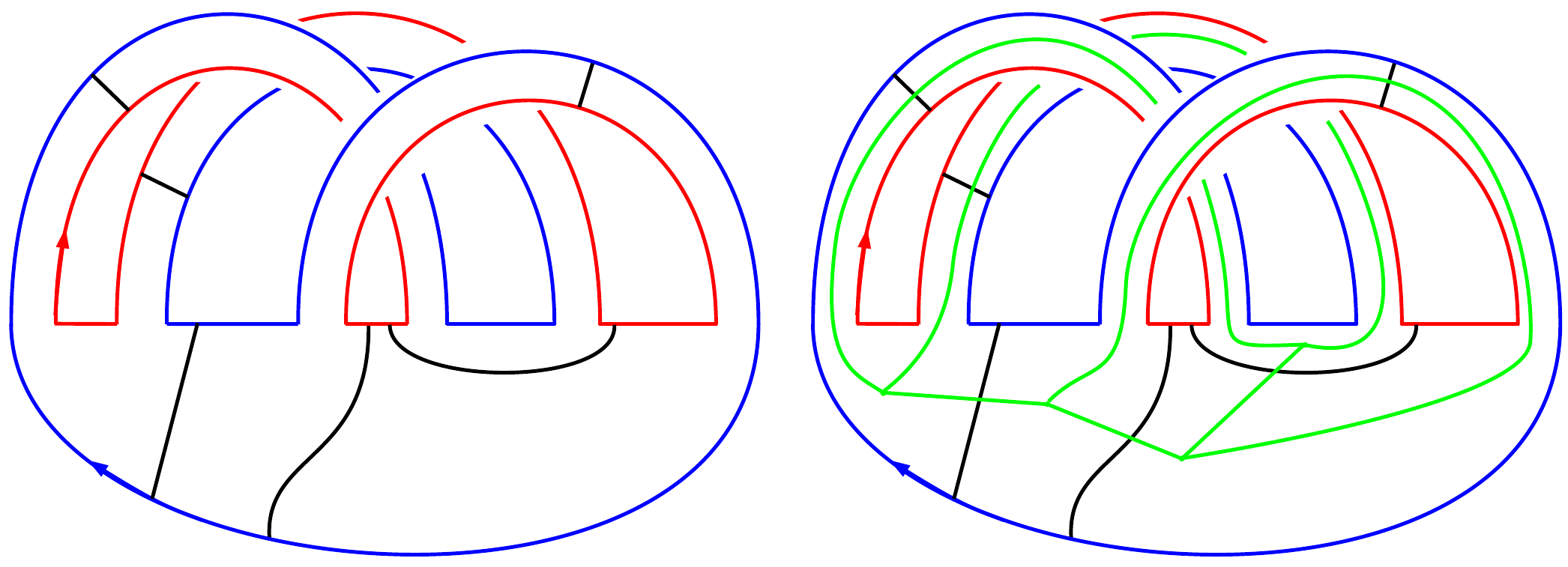}
\caption{{\small  The left illustration is a genus one surface with two boundary curves---coded red and blue.  $C$ is a maximal
collection of $6$ weighted arcs.  The weights, $w_1 , w_2 , w_3 , w_4 , w_5 , w_6$, are non-negative integers. In Theorem \ref{theorem: DWH-G2-D4-I12}, all weights can be taken to be equal to 2 for $\beta$. The green graph
in the right illustration is $G(C)$, the dual graph.  Each edge of $G(C)$ intersects exactly one arc of $C$ once.}}
\label{fig: DWH-G2-2}
\end{figure}
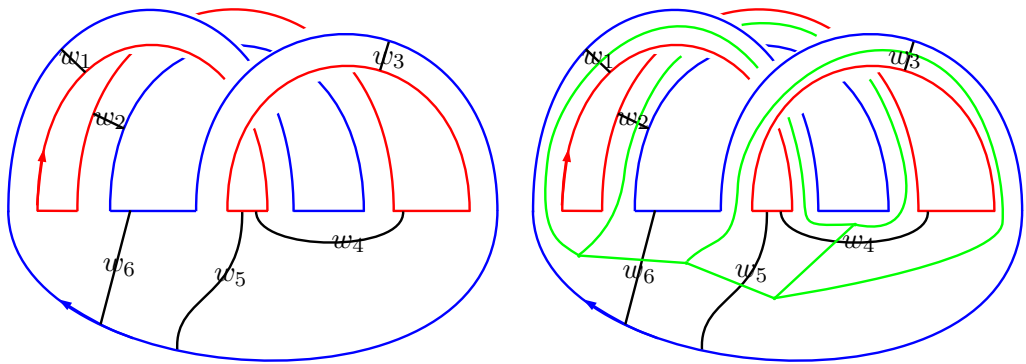

We can translate this above discussion into a system of linear inequalities, one for each such simple closed curve $c$.
First, in $S_{1,2}$ we can collapse all bands of $4$-gon regions to a single properly embedded arc.  This gives us a collection of non-parallel properly embedded essential arcs in $S_{1,2}$.
If need be we throw in additional essential arcs such that our collection, $C$, is a maximal collection of non-parallel properly embedded essential arcs in $S_{1,2}$.
Such a maximal collection splits $S_{1,2}$ up into four $6$-gon disc regions.  (Up to homeomorphism there are only a finite number of ways gluing together four $6$-gons
to construct $S_{1,2}$.)  The left illustration in Figure \ref{fig: DWH-G2-2} shows one possible $6$-gon decomposition of $S_{1,2}$.  To find all possible loops, $c$, of the required
type that intersect any essential arc in our collection at most once, we consider the dual graph, $G(C)$, to this decomposition, as illustrated in the right of Figure \ref{fig: DWH-G2-2}.
One of our needed curves is then just a circuit---a closed edge-path---in $G(C)$. To each such circuit we associate an inequality as follows.
 
First, assign weights, $w_1, w_2,w_3,w_4,w_5,w_6$, to each of the six essential arcs in the decomposition.  Each $w_i$ corresponds to the number of
parallel arcs of $\b$.  Next, for any circuit add together all the weights of the arcs of $C$
that the circuit intersects.  For a circuit $c \subset G(C)$, this sum is equal to $|c \cap \b|$.  Thus, by our previous discussion this sum is greater than $4$ when $d (\a, \b) =4$.
The circuits of Figure \ref{fig: DWH-G2-2} yield the following LIP system (\ref{LLP-1}).

\begin{eqnarray} \label{LLP-1}
\begin{tabular}{r c r c r c r c r c r}
$w_1$ & $+$ & $w_4$ & $+$ & $w_5$ & $+$ & $w_6$ & $\geq$ & $4$ \\
$ w_2$ & $+$ & $w_4$ & $+$ & $w_5$ & $+$ &$w_6$ & $\geq$ & $4$ \\
$\ $       & $ \ $& $\ $    & $ \ $ & $w_3$ & $+$ & $w_5$ & $\geq$ & $4$ \\
$ \ $ & $\ $ & $ \ $ & $\  $ & $w_1$ &$+$ & $w_2$ & $\geq$ & $4$ \\
$ w_1$ & $+$ & $w_3$ & $+$ & $w_4$ & $+$ & $w_6$ & $\geq$ & $4$ \\
$ w_2$ & $+$ & $w_3$ & $+$ & $w_4$ &$+$ & $w_6$ & $\geq$ & $4$ \\
$\ $       & $w_1 ,$& $w_2 , $    & $w_3 , $ & $w_4 , $ & $w_5 ,$ & $w_6$ & $\geq$ & $0$
\end{tabular}
\end{eqnarray}

When we minimize $P(w_1, \cdots , w_6) = \sum w_i $ constrained by LIP (\ref{LLP-1}), we find that the minimum value of $P$ is $8$, which is {\bf twice} $4$---the {\em scaling
factor of $P$ constrained by LIP (\ref{LLP-1})}. The minimum value of $P$ is achieved when $w_1=w_2=w_3=w_5=2$ and $w_4=w_6=0$.
(One can utilize https://www.easycalculation.com/ operations-research/simplex-method-calculator.php online.)

\begin{remark}
In order to re-glue the two boundary curves of $S_{1,2}$ so that in $S_2$ we obtain the curve $\b$, it is necessary that the sum of the weight on each boundary
curve to be equal---an added constraint to (\ref{LLP-1}) that is currently missing.  With such an added constraint, as noted in \cite{GMMM}, any other $6$-gon decomposition of $S_{1,2}$ will produce an LIP that is equivalent to LIP (\ref{LLP-1}) up to relabeling.  So the minimization
of $P$ will still result in a value of $8$.
\end{remark}

By the linearity of LIP (\ref{LLP-1}), if we replaced every occurrence of the lower bound of $4$ by $1$ and asked what would the minimal value of
$P$ so constrained, our answer would be $2$.  Now the minimal intersection for a distance $4$ filling pair as stated in Theorem \ref{theorem: DWH-G2-D4-I12}
is $12$, not $8$.  This value was obtained by doing a search for distance $4$ filling pair utilizing the MICC program \cite{MICC} that is an implementation of
the efficient geodesic algorithm and this search was simplified by starting with filling pairs having at least $8$ intersections.  Knowing this additional
fact we can conclude that any distance $5$ filling pair of $S_2$ must have intersection number at least $ 2 \times 12$.  And, in general
we have our claimed inequality $$ 2^{d-4} \cdot 12 \leq | \a \cap \b | $$ for a filling pair of distance $d$.

Finally, we consider a filling pair, $(\a , \b)$ for a higher genus $S_g$ and we split along $\a$ to produce $S_{(g-1) , 2}$, we observe that
there will always be an embedded $S_{1 ,2}$ in $S_{(g-1) , 2}$, for $g \geq 3$.  Thus, when we produce a $6$-gon decomposition for $S_{(g-1) , 2}$
coming from a complete collection of non-parallel essential properly embedded arcs, the associated LIP for the dual graph will contain (up to relabeling)
a copy of LLP (\ref{LLP-1}). It follows that the scaling factor of $P$---the sum of all the weights---is at least $2$.  This observation plus the fact
that $2g-1$ is the minimal intersection number for a filling pair yields the inequality:
$$ 2^{d - 3} \cdot (2g-1) \leq |\a \cap \b|.$$
\end{proof}

\section{One-vertex triangulation and parallel arcs}
\label{section: One-vertex triangulation and parallel arcs}
Suppose $v_0, v_1, \cdots, v_{\text{d}}$ is a geodesic of length $d \geq 3$ in $\mathcal{C}(S_g)$, and $\a_0,\a_1, \a_{\text{d}}$ are representatives of $v_0, v_1, v_{\text{d}}$ that are pairwise in minimal position. 
For any reference arc $\g$ of the pair $(\a_0, \a_d)$, that is, an arc that is in minimal position with $\a_1$ and whose interior is disjoint from $\a_0 \cup \a_d$. In fact, it suffices to look at the reference arc that connects the midpoints of $\a_0$-edges in a non-rectangular polygon of $S_g\setminus (\a_0  \cup \a_d)$. In this section, we will show the existence of parallel arcs of $\a_1 \setminus \g$ given that the intersection between $\a_1$ and $\g$ is sufficiently large. 
We start off with a calculation of one-vertex triangulation of a closed surface $S_g$. 

\begin{lemma}
 The number of edges in one-vertex triangulation of a closed oriented surface $S_g$ is $6g-3$.
\end{lemma}

\begin{proof}
 Consider a one-vertex triangulation of $S_g$ with number of vertices, edges and faces denoted by $V$, $E$ and $F$, respectively.
 By Euler characteristic calculation, $V-E+F=2-2g$. In the triangulation, each face has 3 edges and each edge is shared by 2 faces, then we have $2E=3F$. Since there is a unique vertex, if we multiply the Euler characteristic by 3, then 
 
\begin{equation}
\label{Euler calculation}
\begin{array}{r@{}l}
  3V-3E+3F &= 6-6g,\\
  3V-3E+2E &= 6-6g,\\
  3V-E &= 6-6g,\\
  3-E &= 6-6g,\\
  E &= 6g-3.
\end{array}
\end{equation}

\end{proof}

Now focus on the components of $\a_1 \setminus \g \subset S_g \setminus \a_0 $.  Two arc components $ c_1 , c_2 \subset \a_1 \setminus \g $
are {\em parallel} in $S_g \setminus \a_0 $ if they are two opposite sides of a rectangular disc---the other two sides being in $\g$. (See Figure~\ref{parallel arcs}.)

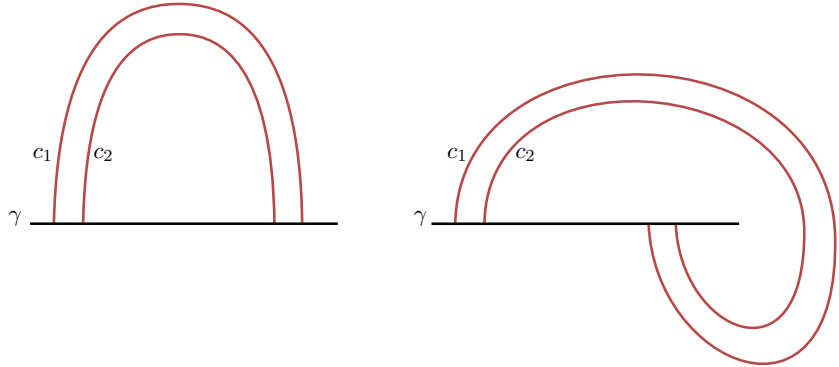
\begin{figure}[htbp]
\labellist
\small\hair 2pt
\pinlabel $\g$ at 0 85
\pinlabel $c_1$ at 15 120
\pinlabel $c_2$ at 48 120
\pinlabel $\g$ at 220 85
\pinlabel $c_1$ at 240 120
\pinlabel $c_2$ at 277 120
\endlabellist
\centering
\scalebox{.80}{\includegraphics[width=0.90\textwidth]{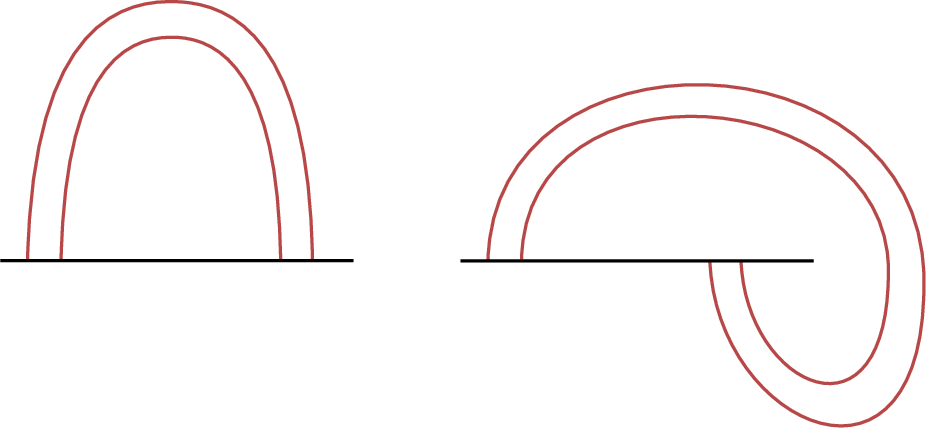}}
\caption{\small  Two configurations of parallel arcs.}
\label{parallel arcs}
\end{figure}
 
 We want to determine the minimal number $C_\sharp$ of components of $\a_1 \setminus \g$ such that once the number of components is larger than $C_\sharp$, then the parallel arcs will occur. When we split $S_g$ along $\a_0$, the reference arc $\g$ is a properly
embedded arc with its endpoints on the boundary of the resulting surface, $S_{\hat{g}}$, where $\hat{g} \leq g - 1$. 

\begin{lemma}
\label{lemma: parallel arcs}
 With the setup on a closed surface $S_g$ as above, if $|\a_1 \cap \g| >  6g-8$, then there exist parallel arcs of $\a_1 \setminus \g$. More generally, if  $|\a_1 \cap \g| >  (k-1)(6g - 8)$, then there exist at least $k$ parallel arcs of $\a_1 \setminus \g$.
\end{lemma}

\begin{proof}
 
There are three cases we need to consider:
\begin{enumerate}
  \item $\a_0$ is non-separating and two endpoints of $\g$ are on the two distinct boundary components of the connected surface $S_g \setminus \a_0$;
  \item $\a_0$ is non-separating and both endpoints of $\g$ are on the same boundary component of the connected surface $S_g \setminus \a_0$;
  \item $\a_0$ is separating and two endpoints of $\g$ is on the single boundary component of one of the two components in $S_g \setminus \a_0$. 
\end{enumerate}

\begin{figure}[htbp]
\centering
\scalebox{.80}{\includegraphics[width=0.90\textwidth]{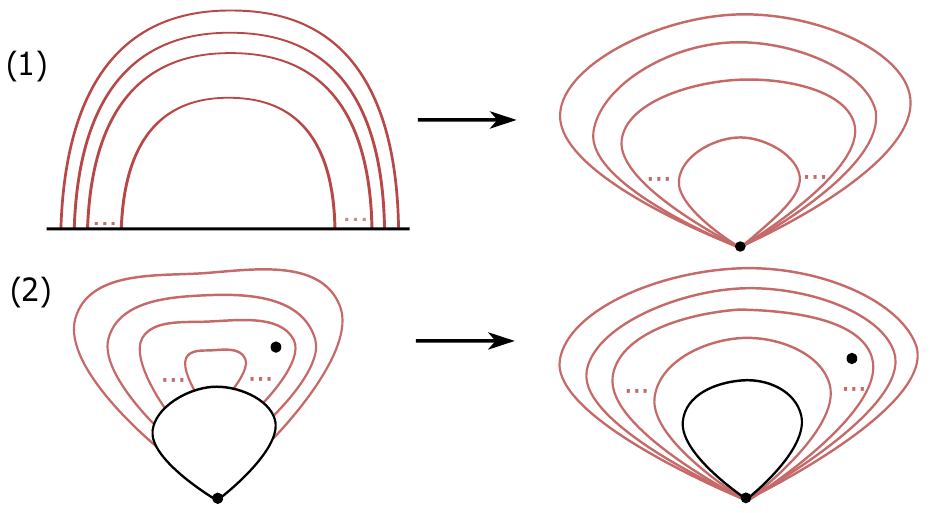}}
\caption{\small  In (1), the reference arc $\g$ in black with two endpoints on two distinct boundary components (two distinct vertices after being crushed) of $S_g \setminus \a_0$ is collapsed to a point; In (2), two boundary components of $S_g \setminus \a_0$ are crushed to two points. The two endpoints of the reference arc $\g$ in black are identified to be a single endpoint, and the intersection points with $\a_1$ are moved along $\g$ to the single endpoint.  }
\label{arc_dot1}
\end{figure}

\begin{figure}[htbp]
\centering
\scalebox{.80}{\includegraphics[width=0.90\textwidth]{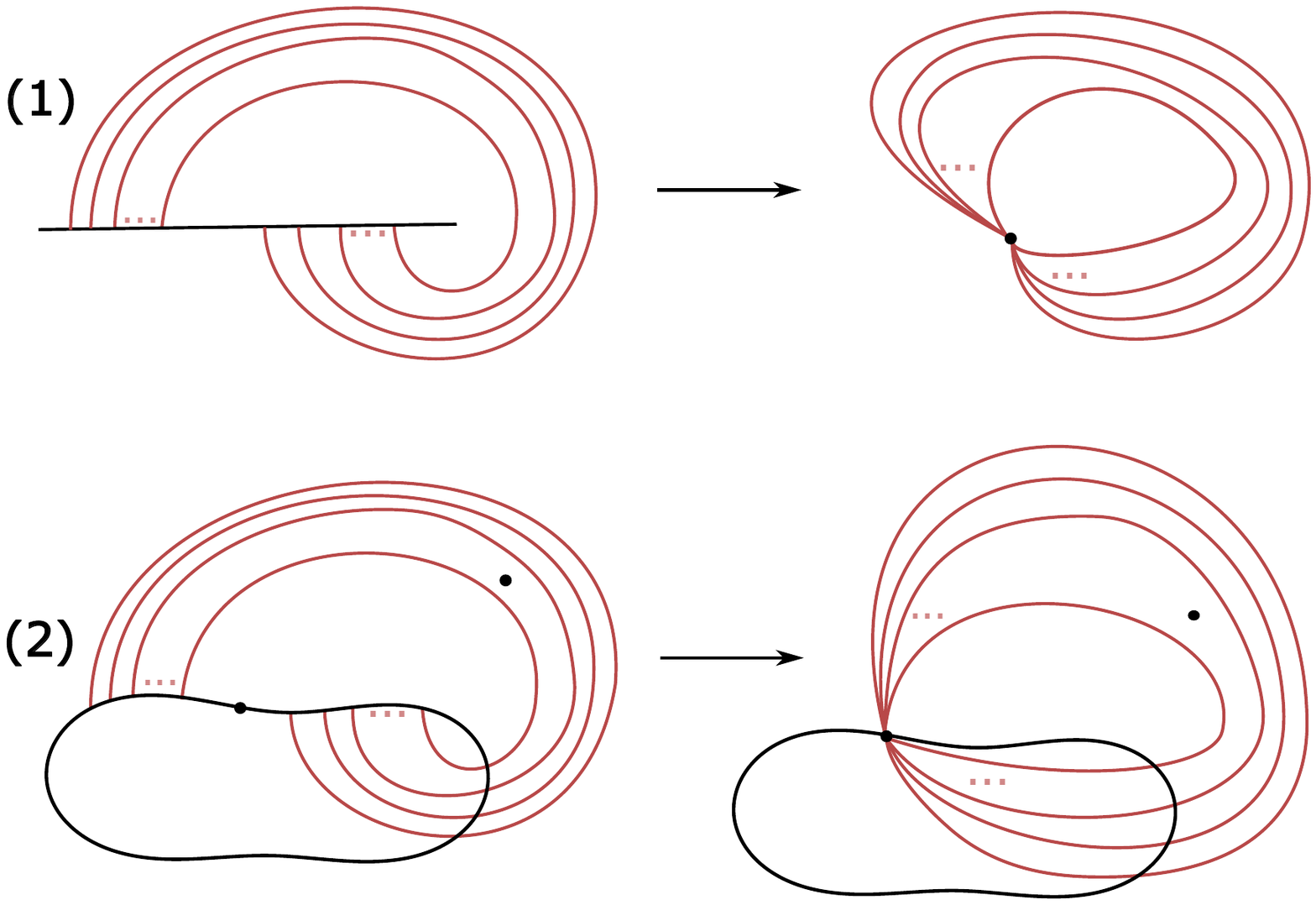}}
\caption{\small  The other configuration is similar to Figure \ref{arc_dot1}. }
\label{arc_dot2}
\end{figure}

In all cases we can think of
$\g$ as being in a surface of at least one less genus with either one or two vertices (the boundary curves crushed to points). 
 The Figures~\ref{arc_dot1} \& \ref{arc_dot2} are illustrated for the two cases in which the two intersection points of an arc are either of opposite orientations or of same orientation. In general, the two cases are mixed and the arcs are not necessarily parallel.
 
 In Figures~\ref{arc_dot1}(1) and \ref{arc_dot2}(1), the two endpoints of the arc $\g$ are on two distinct boundary components of $S_g \setminus \a_0$. The two boundary components are crushed to two points, and the arc $\g$ is collapsed to a single point as well. It follows that the number of components of $\a_1 \setminus \g$ is same as the number of edges in one-vertex triangulation of surface $S_{g-1}$. By Lemma \ref{Euler calculation}, it is $C_\sharp = 6 (g-1) - 3 = 6g - 9$.

\begin{figure}[htbp]
\vspace{-0cm}
\centering
\scalebox{.80}{\includegraphics[width=0.80\textwidth]{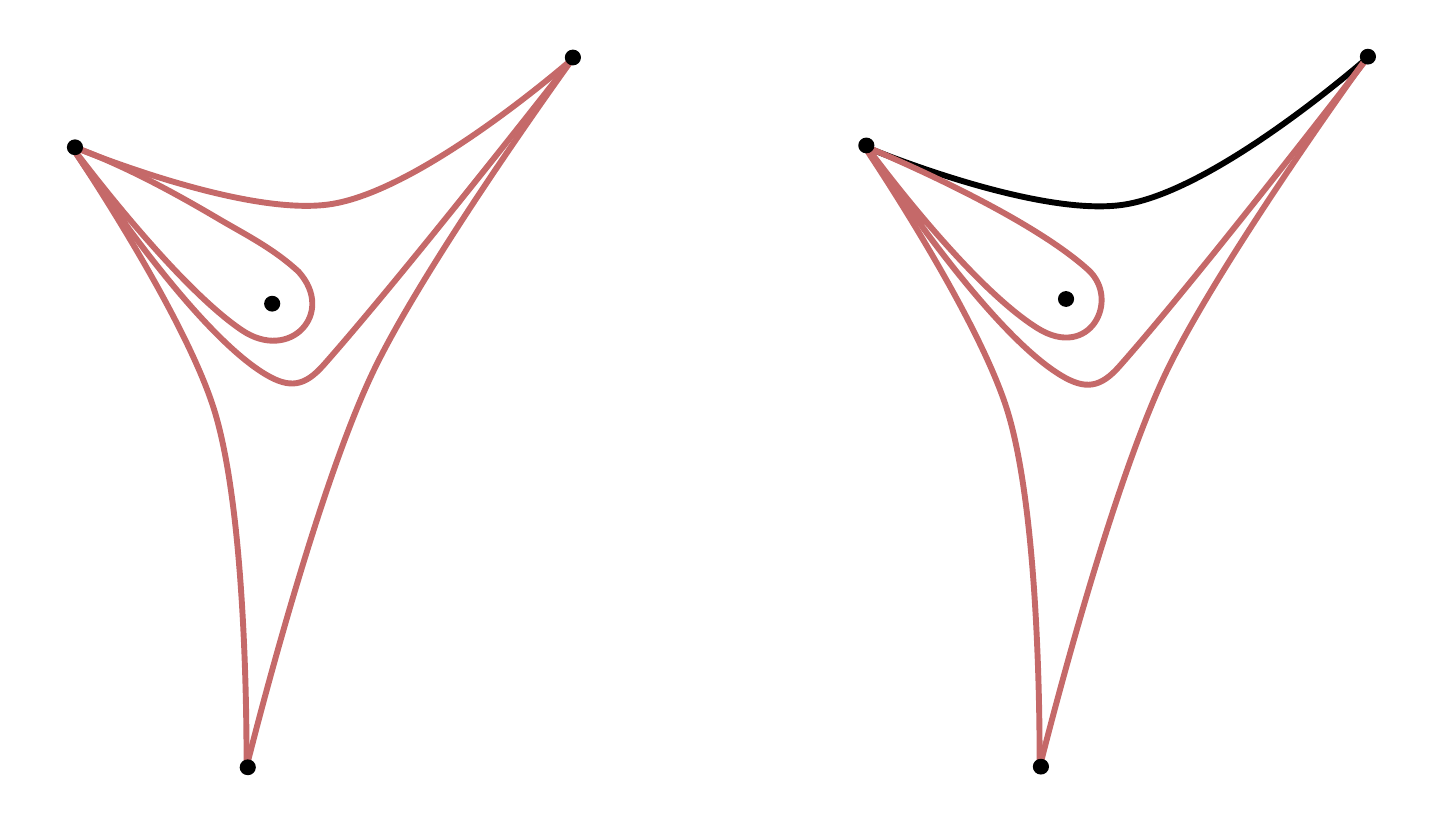}}
\vspace{-0cm}
\caption{\small  The triangles with one puncture are taken from the triangulation in the bottom right illustrations of Figures~\ref{arc_dot1} and \ref{arc_dot2}, in which the vertices of the triangles are identified to be one vertex.  At most two more edges (two classes of parallel components of $\a_1 \setminus \g$) can be added to the triangle. }
\label{triangulation}
\end{figure}

In the second case, Figures~\ref{arc_dot1}(2) and \ref{arc_dot2}(2), both endpoints of $\g$ are on the same boundary/vertex of
$S_{\hat{g}}$. The reference arc $\g$ will be one of the edges of a one-vertex triangulation of a once punctured $S_{g-1}$. Since two more non-parallel edges are allowed to be added to the triangle containing the vertex crushed by the other boundary as illustrated in Figure~\ref{triangulation}. The number of components of $ \a_1 \setminus \g $ is
$C_\sharp = 6 (g-1) - 3 - 1 + 2 = 6g - 8$.

In the third case, the number of components of $\a_1 \setminus \g$  surpasses the second case. Hence, in all cases, the most possible number of non-parallel components of $\a_1 \setminus \g$ is $C_\sharp = 6g - 8$. 

The number of components of $\a_1 \setminus \g$ is the same as $|\a_1 \cap \g| $. If $|\a_1 \cap \g| > C_\sharp$, then there must be at least two arcs of $ \a_1 \setminus \g $ that are parallel.  More generally, if $ |\a_1 \cap \g| > (k-1) C_\sharp  = (k-1)(6g-8)$, there must be at least $k$ parallel arcs by the Pigeonhole principle. 

\end{proof}

\section{Efficient geodesics}
\label{section: efficient geodesics}

In this section, we will recall some fundamentals of the efficient geodesics from \cite{BMM}. Let $v$ and $w$ be vertices of $\mathcal{C}(S_g)$ with $d(v,w)=d \geq 3$ and $v=v_0, v_1, \cdots, v_{\text{d}} = w$ be a geodesic connecting $v$ to $w$. The intersection between the representatives $\a_i$ of $v_i$ and the reference arc $\g$ produces a sequence of natural numbers along $\g$, which is called the \emph{intersection sequence} of the $\a_i$ along $\g$. 

An intersection sequence $\sigma$ of natural numbers $(j_1, j_2, \cdots, j_k)$ can be arranged in a normal form called \emph{sawtooth form}, that is, 

$$j_i \leq j_{i+1} \Longrightarrow j_{i+1} =  j_i + 1.$$

If the intersection sequence of the curves along $\g$ in sawtooth form is viewed as a function, $\{1, 2, \cdots, N\} \rightarrow \mathbb{N}$, where $N$ is the cardinality of $\g \cap (\a_1 \cup \a_2 \cup \cdots \cup \a_{d-1})$, then the graph is a set of lattice points of integer coordinates. The graph of the sequence is called \emph{dots}, and the line segments resulting from the join of dots with slope 1 are called \emph{ascending segments}. The resulting graph of intersection sequence $\sigma$ in sawtooth form is called the \emph{dot graph}, denoted by $G(\sigma)$. An example of dot graph is illustrated in Figure \ref{DotGraph}.

\begin{figure}[htbp]
\centering
\includegraphics[width=0.80\textwidth]{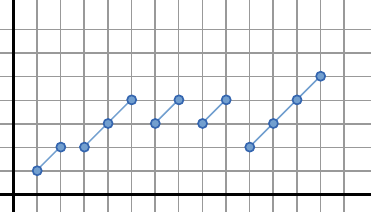}
\caption{\small A typical dot graph of intersection sequence in sawtooth form. }
\label{DotGraph}
\end{figure}

Next, we will deal with certain shapes of polygons in the dot graph.
A polygon in the plane is a \emph{dot graph polygon} if
\begin{enumerate}
\item the edges all have slope 0 or 1,
\item the edges of slope 0 have nonzero length, and 
\item the vertices all have integer coordinates.
\end{enumerate}
The edges of slope 1 in a dot graph polygon are called \emph{ascending edges} and the edges of slope 0 are called \emph{horizontal edges}.

Let $\sigma$ be a sequence of natural numbers in sawtooth form.  A dot graph polygon is a \emph{$\sigma$-polygon} if:
\begin{enumerate}
 \item the vertices are dots of $G(\sigma)$ and
 \item the ascending edges are contained in ascending segments of $G(\sigma)$. 
\end{enumerate}

\begin{figure}[htbp]
\centering
\includegraphics[width=0.90\textwidth]{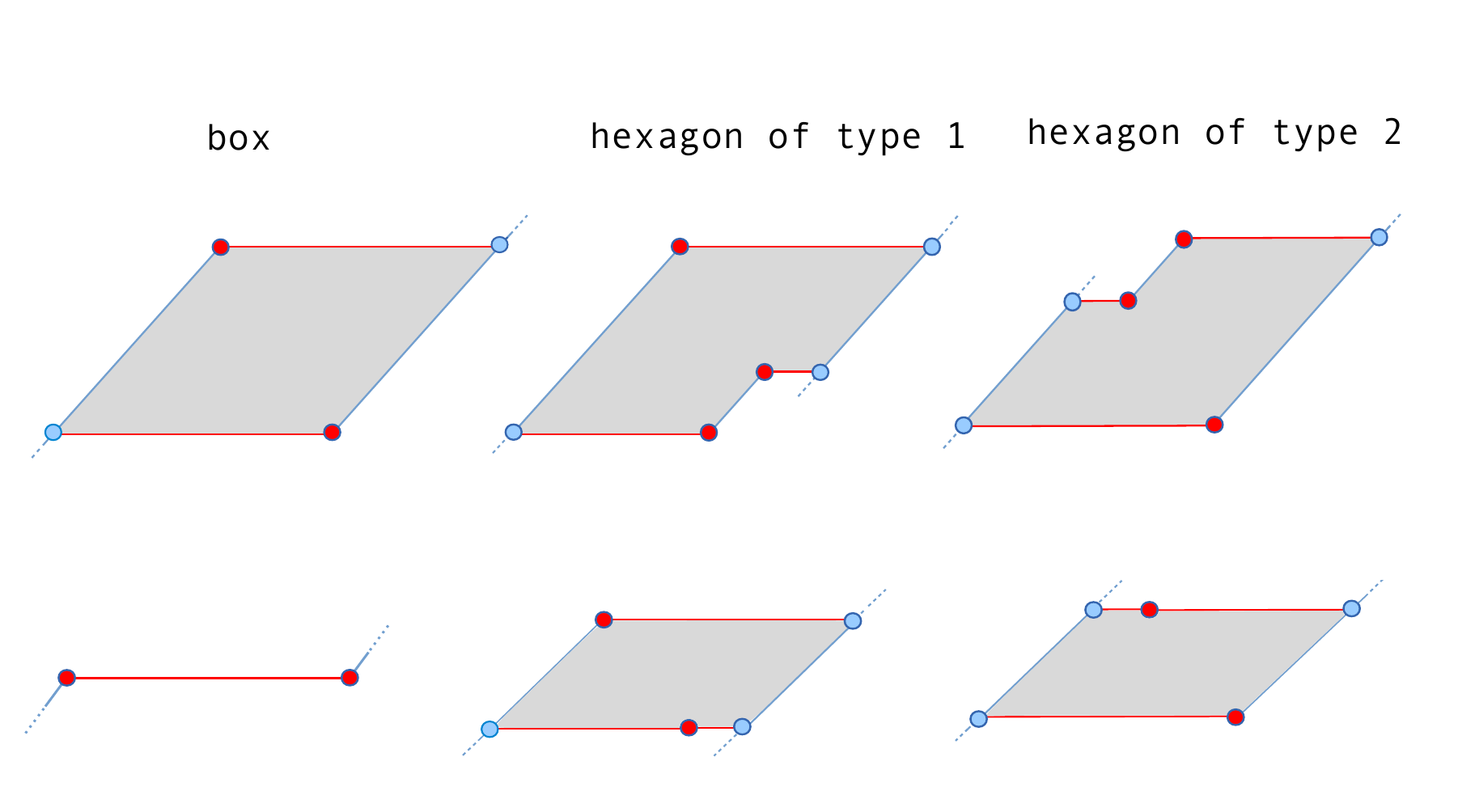}
\caption{\small Box and hexagons of types 1 and 2 on the top; Corresponding degenerate box and hexagons at the bottom. Red dots are endpoints of ascending segments, and blue dots may or may not be endpoints. }
\label{BoxHexagons}
\end{figure}

A \emph{box} in $G(\sigma)$ is a $\sigma$-quadrilateral $P$ with the following two properties:
\begin{enumerate}
 \item the leftmost ascending edge contains the highest point of some ascending segment of $G(\sigma)$ and
 \item the rightmost ascending edge contains the lowest point of some ascending segment of $G(\sigma)$.
\end{enumerate}

Up to translation and changing the edge lengths, there are four types of dot graph hexagons; two have an acute exterior angle, and we will not need to consider these.  Notice that a dot graph hexagon necessarily has a leftmost ascending edge, a rightmost ascending edge, and a middle ascending edge.  This holds even for degenerate hexagons since horizontal edges are required to have nonzero length.

A \emph{hexagon of type 1} in $G(\sigma)$ is a $\sigma$-hexagon where:
\begin{enumerate}
\item no exterior angle is acute, 
\item the middle ascending edge is an entire ascending segment of $G(\sigma)$, and 
\item the minimum of the middle ascending edge equals the minimum of the leftmost ascending edge, 
\item the leftmost ascending edge contains the highest point of an ascending segment of $G(\sigma)$.
\end{enumerate}

Similarly, a \emph{hexagon of type 2} in $G(\sigma)$ is a $\sigma$-hexagon that satisfies the first two conditions above and the following third  and fourth conditions:
\begin{enumerate}
\item[($3'$)] the maximum of the middle ascending edge equals the maximum of the rightmost ascending edge,
\item[($4'$)]  the rightmost ascending edge contains the lowest point of an ascending segment of $G(\sigma)$.
\end{enumerate}

See Figure~\ref{BoxHexagons} for pictures of boxes and hexagons of types 1 and 2 and their degenerate cases.

\begin{figure}[htbp!]
\vspace{0.5cm}
\labellist
\small\hair 2pt
\pinlabel $3$ at 15 250
\pinlabel $4$ at 47 250
\pinlabel $5$ at 79 250
\pinlabel $3$ at 110 250
\pinlabel $4$ at 144 250
\pinlabel $5$ at 175 250
\pinlabel $3'$ at 359 252
\pinlabel $4'$ at 391 252
\pinlabel $5'$ at 517 252
\pinlabel $5$ at 150 83
\pinlabel {\tiny $-+$} at 285 91
\pinlabel $4$ at 150 55
\pinlabel {\tiny $++$} at 265 63
\pinlabel $3$ at 150 25
\pinlabel {\tiny $+-$} at 244 33
\endlabellist
	\centerline{\includegraphics[width=0.95\textwidth]{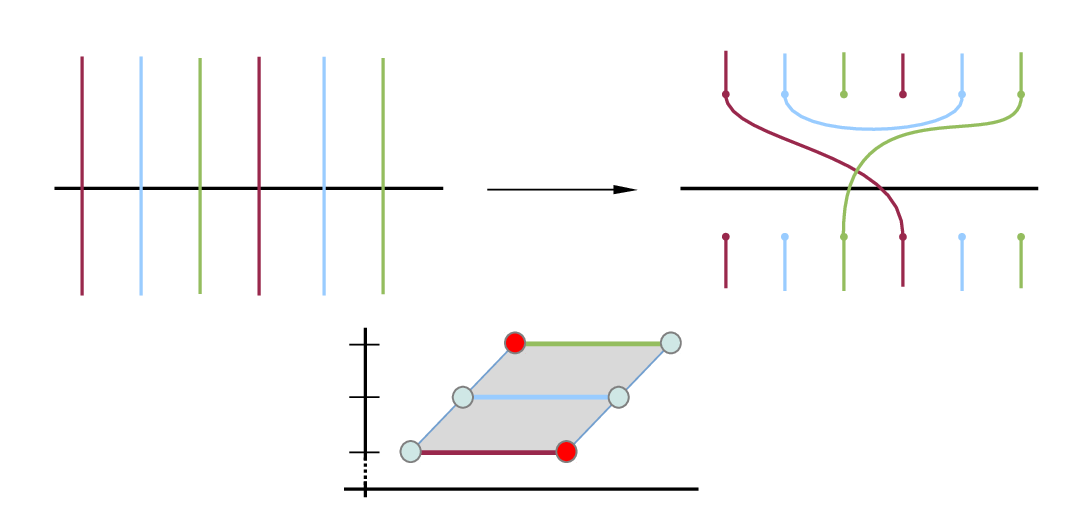}}
	\caption{{\small An example of a set of surgeries as in the box case.}}
	\label{Box}
\end{figure}
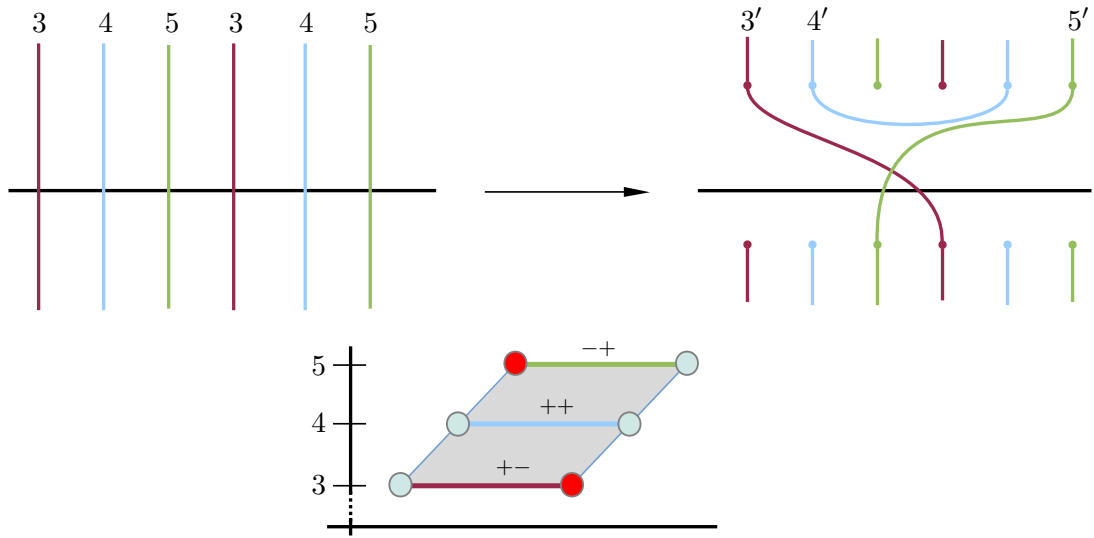

\begin{figure}[htbp!]
\vspace{0.5cm}
\labellist
\small\hair 2pt
\pinlabel $3$ at 27 281
\pinlabel $4$ at 42 281
\pinlabel $5$ at 57 281
\pinlabel $6$ at 72 281
\pinlabel $7$ at 87 281
\pinlabel $3$ at 101 281
\pinlabel $4$ at 115 281
\pinlabel $5$ at 129 281
\pinlabel $2$ at 144 281
\pinlabel $3$ at 158 281
\pinlabel $4$ at 172 281
\pinlabel $5$ at 186 281
\pinlabel $6$ at 200 281
\pinlabel $7$ at 214 281
\pinlabel $5'$ at 361 281
\pinlabel $6'$ at 375 281
\pinlabel $3'$ at 404 281
\pinlabel $7'$ at 520 281
\pinlabel \begin{rotate}{-20}\textcolor{red}{\Large $\boldsymbol{\times}$}\end{rotate}  at 480 265
\pinlabel $4'$ at 374 172
\pinlabel {\tiny $\mplus$} at 292 115
\pinlabel $7$ at 122 111
\pinlabel {\tiny $++$} at 280 98
\pinlabel $6$ at 122 93
\pinlabel {\tiny $\pminus$} at 229 80
\pinlabel $5$ at 122 74
\pinlabel \begin{rotate}{-20}\textcolor{red}{\Large $\boldsymbol{\times}$}\end{rotate} at 329 72
\pinlabel {\tiny $\mminus$} at 217 63
\pinlabel $4$ at 122 57
\pinlabel {\tiny $\mplus$} at 198 46
\pinlabel $3$ at 122 40
\pinlabel $2$ at 122 24
\endlabellist
	\centerline{\includegraphics[width=1.0\textwidth]{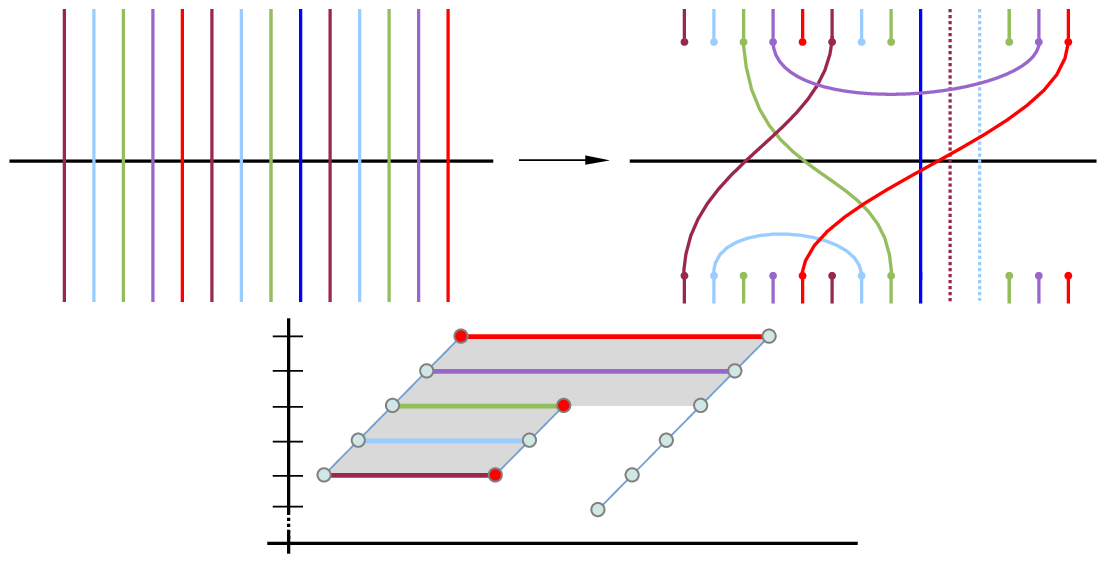}}
	\caption{{\small An example of a set of surgeries as in the hexagon case.}}
	\label{Glock}
\end{figure}

As mentioned in the \S \ref{main results}, given a path $p$ with vertices $v=v_0, v_1, \cdots, v_{\text{d}}=w$ with $d \geq 3$ in $\mathcal{C}(S_g)$, the {\em{complexity}} of the path $p$ is defined to be $$\kappa(p) = \sum_{k=1}^{d-1} (i(v_0, v_k) + i(v_k, v_{\text{d}})).$$

We will show that a geodesic of minimal complexity is initially super efficient. If $d-1 \leq 15 \cdot (6g-8)$, the initial efficiency of this short geodesic is established by the existence of initially efficient geodesics \cite{BMM}. The obstruction of a geodesic being minimal complexity is the occurrence of these three types of dot graph polygons. If such dot graph polygons occur, we can perform the corresponding curve surgeries to reduce the complexity of the geodesic. Here are some examples for the box surgery, Figure~\ref{Box}, and hexagon surgery of type 1, Figure~\ref{Glock}. 
The outline of the proof is the constraint on the distance (i.e, $d-1 \geq 15 \cdot (6g-8) + 1$) eventually makes the geodesic of minimal complexity to have infinite length. The difficulty is to obtain an explicit bound of super initial efficiency that only depends on the genus of a surface. Ultimately, we will argue that the explicit bound is $15 \cdot (6g-8)$ for $g \geq 3$ and $44$ for $g = 2$.

\section{Obstruction to surgeries}
\label{section: Obstruction to surgeries}

Let $v=v_0, v_1, \cdots, v_{\text{d}}=w$ be a geodesic of minimal complexity with sufficiently large distance (i.e., $d \geq 15\cdot(6g-8)+1$ for $g \geq 3$, and $d \geq 45$ for $g=2$), and let $\a_0$, $\a_1$ and $\a_d$ be the representatives of $v_0, v_1$ and $v_{\text{d}}$ that intersect minimally without any triple points.  By Lemma \ref{lemma: parallel arcs}, we know that if the intersection number $|\a_1 \cap \g|$ is larger than $6g - 8$, then it will create some parallel arcs of $\a_1 \setminus \g$. More generally, if the intersection number $|\a_1 \cap \g|$ is larger than $(k-1)(6g-8)$, then it will create at least $k$ parallel arcs of $\a_1 \setminus \g$. Suppose the number of the parallel arcs of $\a_1 \setminus \g$ is $k$. In the following, let us focus on the left end of the $k$ parallel arcs intersecting with the reference arc $\g$, which gives rise to the leftmost $k$ intersection points of $\a_1 \cap \g$. See Figure \ref{two_stack} for an example when $k = 3$. Assume the intersection sequence on $\g$ is in sawtooth form, and intersection subsequence trapped in the $k$ intersection points of $\a_1 \cap \g$ determines $k$ ascending segment starting at index 1 in the dot graph, possibly with other ascending segments in the middle. Recall that our goal is to show the geodesic has infinite length eventually, so we will argue the worst scenario that the highest index of the curve in the $k$ ascending segments is minimal. In this way, we will increment the length of the geodesic segment minimally. Let $e_1, e_2, \cdots, e_k$ be the $k$ ascending segments in the dot graph, and let $\max(e_i)$ be the highest index of the curve in $e_i$. We have the following result to describe the shape of dot graph.

\begin{lemma}
\label{lemma: triangular dot graph}
Suppose that the highest index of the curve in the $k$ ascending segments is minimal. Then $\max(e_i) - \max(e_{i+1}) = 1$ for $1 \leq i \leq k-1$, and  $\max(e_k) = 1$. It implies that the curve of highest index is in the leftmost ascending segment, $e_1$.
\end{lemma}

\begin{proof}
   Assume that $\max(e_i) \leq \max(e_{i+1})$ for some $i$. It follows that the geodesic does not have smallest complexity using the proof of Proposition 3.1 in \cite{BMM}. That implies, $\max(e_i) > \max(e_{i+1})$ for all $i$. The minimum of $\max(e_1)$ forces $\max(e_i) - \max(e_{i+1}) = 1$ and $\max(e_k) = 1$. Then the highest index of $e_1$ is $\max(e_1) = k$. 
\end{proof}

This triangular shape dot graph has the following explanation in the parallel arcs of $\a_1 \setminus \g$. First, notice that by efficiency every rectangular disc
illustrating the parallel nature of two arcs of $\a_1 \setminus \g$  must contain/trap an arc of $\a_2 \setminus \g$. Otherwise, two parallel arcs allows a box surgery, so the complexity is not minimal. Now, consider two such rectangular discs stack together, that is, we have three parallel arcs, $\a^1_1, \a^2_1 , \a^3_1 \subset \a_1 \setminus \g $ arranged with $\a^2_1$
being common to two rectangular discs.  Then by efficiency there are arcs $\a^1_2 , \a^2_2 \subset \a_2 \setminus \g $ such that $\a^1_2 $ is contained in the
rectangular disc having $\a^1_1$ \& $\a^2_1$ on its boundary; and, $\a^2_2$ is contained in the rectangular disc having $\a^2_1$ \& $\a^3_1$ on its boundary.  But, then this
configuration will have $\a^1_2$ and $\a^2_2$ being parallel.  Moreover, by efficiency again we must have an arc of $\a_3 \setminus \g$ contained in the
rectangular disc (which is a sub-disc of the two stacked rectangular discs we started with) that illustrates $\a^1_2$ and $\a^2_2$ are parallel.  (To drive the nail home, without an 
arc of $\a_3 \setminus \g$ we would have a box surgery.)  Figure \ref{two_stack} illustrates this stacking of two rectangular discs configuration and its corresponding triangular dot graph. Similarly, the $k$ parallel arcs of $\a_1$ traps $k-1$ arcs of $\a_2$, and that turns out to trap $k-2$ arcs of $\a_3$, etc.. 

\begin{figure}[htbp]
\centering
\scalebox{.80}{\includegraphics[width=0.90\textwidth]{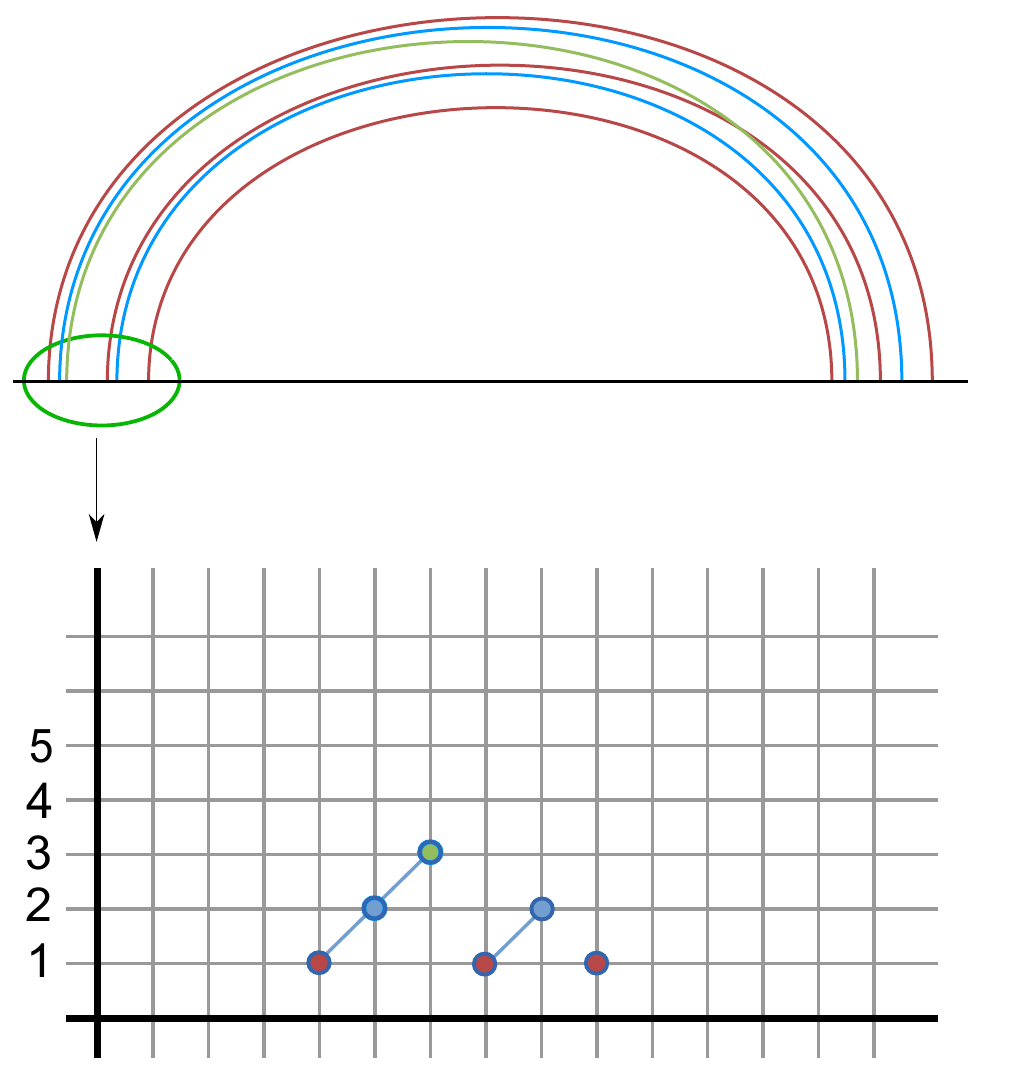}}
\caption{\small The three parallel arcs of $\a_1$ are colored red; the two parallel arcs of $\a_2$ are colored blue; and, the arc of $\a_3$ is colored green. At the bottom, the  triangular dot graph corresponds to the intersection points on the left end of the 3-stack rainbow.}
\label{two_stack}
\end{figure}

\section{Rainbow}
\label{section: rainbow}

The triangular dot graph of the intersection points on the left end in the previous lemma forms a nested rectangular region if we lay out all the parallel arcs according to the indices. We call the nested region a {\em rainbow}, and {\em k-stack rainbow} if it contains $k$ parallel arcs of the outermost boundary curve. If we continue to play this stacking of rectangular game we can force arcs of $\a_i$, with $i$ increasing into a rainbow of parallel arcs.   Figure \ref{five_stack}
illustrates the configuration when we have five parallel arcs of $\a_1 \setminus \g$. As shown in Lemma \ref{lemma: triangular dot graph}, the stacking of five $\a_1$ arcs forces the occurrence of four
$\a_2$ arcs between the $\a_1 {\rm 's}$ by efficiency; three $\a_3$ arcs between the $\alpha_2 {\rm 's}$; and, two $\a_4$ arcs between the $\a_3 {\rm 's}$.
The illustration is actually missing a single arc between the two $\a_4$ arcs (making a box surgery available) since the illustration is becoming cluttered,
but efficiency requires that one should be inserted.

\begin{figure}[htbp]
\centering
\scalebox{.80}{\includegraphics[width=0.90\textwidth]{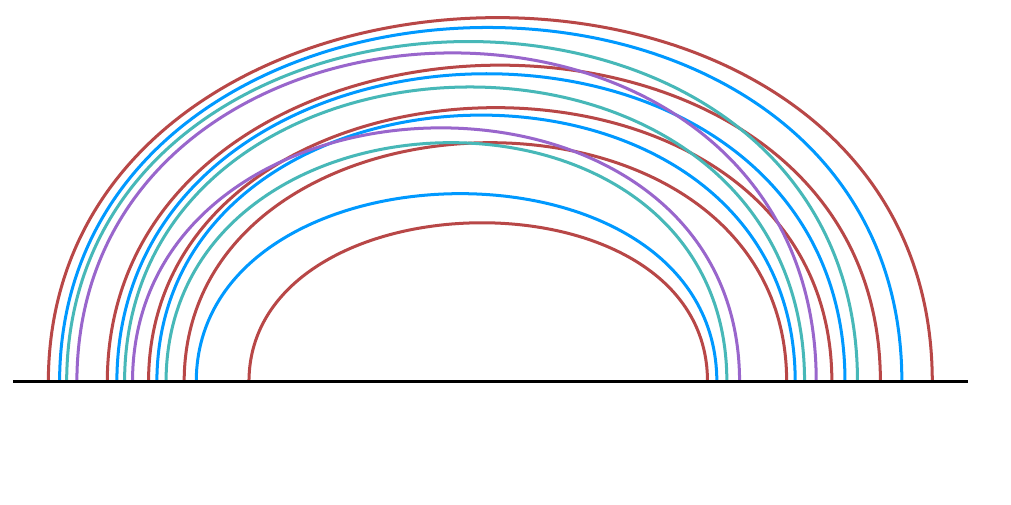}}
\caption{\small  A 5-stack rainbow.  Again parallel arcs of $\a_1$ are colored red; arcs of $\a_2$ are colored blue; arcs of $\a_3$ is colored green; and, the single arc
of $\a_4$ is colored magenta.}
\label{five_stack}
\end{figure}

In this section, we show that if $\a_1 \cap \g$ is sufficiently large, then it gives rise to parallel arcs of $\a_i \setminus \g$  for some curve $\a_i$. Our goal is to obtain such an estimate of the intersection number to make this happen.  This is a first step in our argument to show the process is infinite once we modify the estimate of the intersection number with reference arcs. 

Our argument in a nutshell is that, $k$ parallel arcs of $\a_1$ will force the existence of $k$ parallel arcs of $\a_{i_1}$, for some $i_1 > 1$, which will force
the existence of $k$ parallel arcs of $\a_{i_2}$, for some $i_2 > i_1$, and etc..  Since all parallel arcs for each $\a_{i_j}$ will have their endpoints on $\g$
and since $k$ is fixed, we get a contradiction in that our geodesic is of finite length but we have an infinite sequence of subarcs of  $\a_1, \a_{i_1}, \a_{i_2}, \cdots ,\a_{i_j} , \cdots$.
The core of this argument is calculating a value for $k$.  We start off with a warm-up case.

\begin{proposition}
\label{proposition: calculate k}
Let $v=v_0, v_1, \cdots, v_{\text{d}}=w$ be a sufficiently long geodesic of minimal complexity in $\mathcal{C}(S_{g \geq 2})$, and let $\a_0$, $\a_1$ and $\a_{\text{d}}$ be representatives of  $v_0$, $v_1$ and $v_{\text{d}}$, respectively. There exists an integer $k$ satisfying that, for any reference arc $\g$ of the pair $(\a_0, \a_d)$ such that $\a_1 \setminus \g$ has $k$ parallel arcs, then $|\a_{k+1} \cap \g| \geq k$, where $\a_{k+1}$ is a representative of $v_{k+1}$. More precisely, the minimal value is $k=5$ for $g=2$, $k = 7$ for $g = 3, 4, 5, 6, 7$, and $k=8$ for $g \geq 8$. 
\end{proposition}

\begin{proof}
By Lemma \ref{lemma: triangular dot graph}, if there is a $k$-stack rainbow illustrating $k$ parallel $\a_1$ arcs then by efficiency the rainbow will contain an arc of $\a_k$, that is, the curve of highest index in the leftmost ascending segment. This point is key since
the disjointness of $\a_k$ and $\a_{k+1}$ implies that $\a_{k+1}$ cannot transversely intersect the rainbow. We now perform the obvious surgery that takes a union of an arc in the rainbow of 
$\a_1 \setminus \g$ and a subarc of $\g$ to form loops that intersect $\g$ either
zero (for left configuration in Figure \ref{parallel arcs}) or once (for right configuration in Figure \ref{parallel arcs}).  See Figure \ref{loops}.
Call this new loop $\a_1^\prime$ and notice that it does not intersect $\a_0$.

\begin{figure}[htbp]
\centering
\scalebox{.80}{\includegraphics[width=0.90\textwidth]{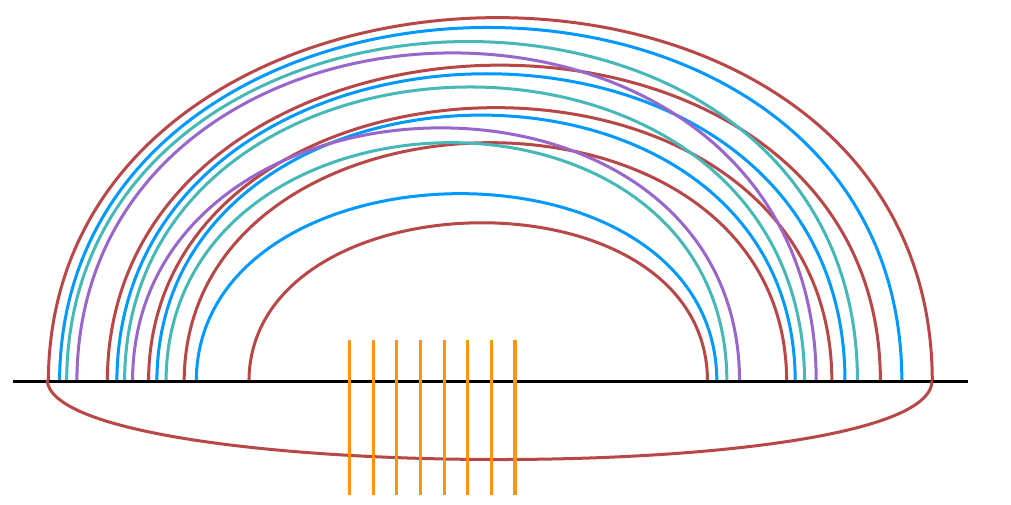}}
\caption{\small  The $\a_1'$ is the union of an outermost arc of $\a_1$ and a subarc of $\g$. The orange arcs intersecting the middle of $\g$ are subarcs of $\a_{k+1}$.}
\label{loops}
\end{figure}

Now consider $|\a_1^\prime \cap \a_{k+1} |$.  Notice that since $\a_{k+1}$ cannot intersect transversely the $k$-stack rainbow, the only place
$\a_1^\prime$ can intersect $\a_{k+1}$ is between the two ``ends'' of the rainbow as illustrated in Figure \ref{loops}.  
Thus, each intersection of $\a_{k+1}$ with $\a_1^\prime$ also corresponds to an intersection of $\a_{k+1}$ with $\g$.  Our strategy is to
``drive up'' the value of $|\a_{k+1} \cap \g|$ so that there are at least $k$ parallel arcs of $\a_{k+1} \setminus \g$.   If so, we might
repeat the above rainbow construction for these parallel $\a_{k+1}$ arcs to say that we must have $\a_{2k + 1}$ intersecting $\g$
enough times to have $k$ parallel arcs.  Then, an iteration of our rainbow construction never ends, which is a contradiction because the
distance between $\a_0$ and $\a_d$ is finite.

With this in mind, we know that $d(\a_1^\prime , \a_{k+1}) \geq k$, because $|\a_0 \cap \a_1'|=0$.  Now using Theorem
\ref{theorem: intersection growth} for $g = 2$,  we have $2^{k-4} \cdot 12 \leq | \a_1^\prime \cap \a_{k+1} | (= |\g \cap  \a_{k+1} |)$ and $C_\sharp = 6(2) - 8 = 4$.  To drive up this
intersection number, we ask the question: for what value of $k$ does
\begin{equation}
\label{genus-2-a1}
 \frac{2^{k - 4}\cdot 12}{C_\sharp}  = \frac{2^{k - 4} \cdot 12}{4}  > k - 1?
\end{equation}
Our initial answer is to have $k=5$. With $|\a_1 \cap \g | >  4 \cdot 4 = 16$,  we will get a rainbow of $\a_1$ arcs that
is at least a $5$-stack rainbow.  We will need to come back to equation \ref{genus-2-a1} later to increase the size of this $\a_1$ rainbow for achieving
a contradiction.

For $g>2$, we have $2^{k-3} (2g -1) \leq | \a_1^\prime \cap \alpha_{k+1} | (= | \g \cap \a_{k+1}|)$.  To drive up this
intersection number, our equation \ref{genus-2-a1} is changed to asking for what value of $k$ does
\begin{equation}
\label{genus-g-a1}
\frac{2^{k - 3} (2g -1)}{C_\sharp} = \frac{2^{k - 3} (2g -1)}{6g-8 } > k - 1?
\end{equation}

When $g = 3, 4, 5, 6, 7$, then $k=7$ works.  For $g \geq 8$, $k=8$ always works.  Again, to guarantee an $8$-stack rainbow for arcs of $\alpha_1$, $| \alpha_1 \cap \g | > 7 \cdot (6g - 8 )$, we will need to come back to this calculation later to increase the size of the $\alpha_1$ rainbow.
\end{proof}

Our strategy now is to argue that when we iterate this construction, having a rainbow forces the existence of another rainbow with the stack number not decreasing, which makes
our original geodesic of minimal complexity have infinite length, then we have a contradiction.  

More precisely, we show that there are curves with higher indices (i.e. $\a_2, \cdots, \a_k$) trapped in the rainbow in Proposition \ref{proposition: calculate k}. On the other hand, once we look at the next rainbow formed by $\a_{k+1} \setminus \g$ as the outermost boundary curve, it is possible that we can extend the curves trapped in the rainbow to the curves with lower indices. Suppose that we fix the curve $\a_{k+1}$, then we will search the subarcs of $\a_j \setminus \g$ with highest index $j$ and the subarcs of $\a_i \setminus \g$ with lowest index $i$ trapped in the rainbow formed by the subarcs of $a_{k+1} \setminus \g$. In this way, the curve $\a_i$ of lowest index will produce a new curve $\a_i'$ that acts as $\a_1'$ from $\a_1$ in the first rainbow. 

The process turns out to be a race between the distance between two curves $\a_i$ (or $\a_i'$),  $\a_{j+1}$ and the intersection number of $\g  \cap \a_{k+1}$.  We are looking for the worst scenario in which the distance is as small as possible while the intersection number is fixed. Same to the first rainbow, our goal is to increment the length of the geodesic segment minimally. In the next round, $\a_{j+1}$ becomes the outermost curve to form a new rainbow, and the process continues forever, which gives rise to a contradiction.

However, there is a subtlety as we examine the rainbow at the next stage.  Let us go to the general case that $g>2$ .  
For the $g > 2$ case and having $k=8$, our calculation gives us that $\a_1^\prime$  intersects $\a_9$ at least $2^{8-3} \cdot (2g-1) = 32 \cdot (2g-1)$ times; and, there must be
at least $11 (=\max ( \{\lceil \frac{32\cdot (2g-1)}{6g-8} \rceil | \text{for any    } g > 2\} ))$-stack rainbow coming from parallel arcs of $\a_9 \setminus \g$. The first observation is that there must be an $\a_{10}$ between the parallel arcs of $\a_9 \setminus \g$.

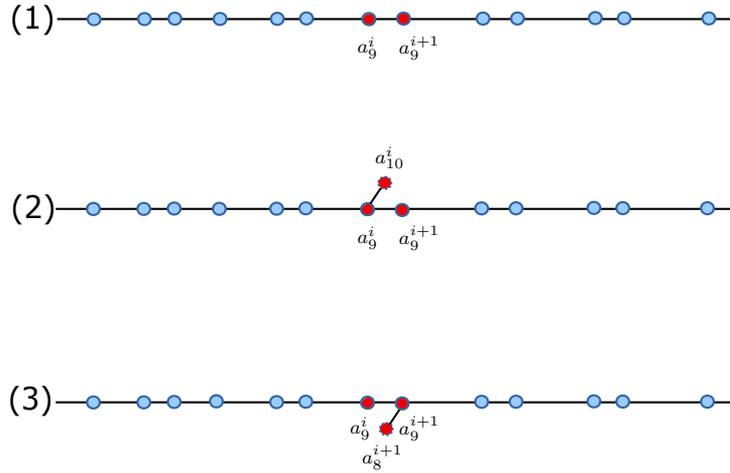
\begin{figure}[htbp!]
\labellist
\tiny \hair 2pt
\pinlabel $a_9^i$ at 405 400
\pinlabel $a_9^{i+1}$ at 440 400

\pinlabel $a_9^i$ at 405 270
\pinlabel $a_9^{i+1}$ at 440 270
\pinlabel $a_{10}^i$ at 420 325

\pinlabel $a_9^i$ at 400 140
\pinlabel $a_9^{i+1}$ at 440 140
\pinlabel $a_{8}^{i+1}$ at 415 120

\endlabellist
\centering
\scalebox{.90}{\includegraphics[width=0.90 \textwidth]{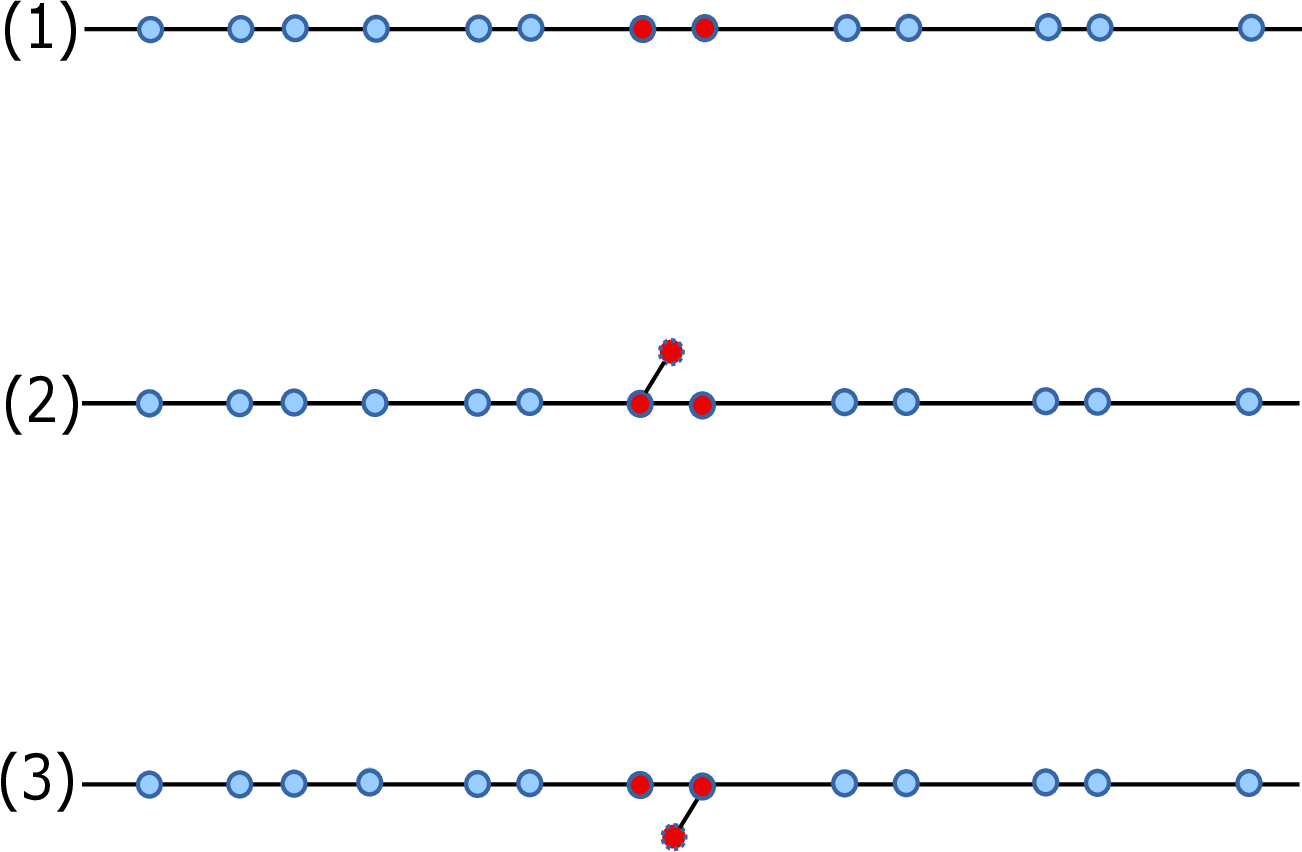}}
\vspace{-1cm}
\caption{\small In (1), two adjacent intersection points $a_9^i$ and $a_9^{i+1}$ are colored red; (2) and (3) illustrate two possible ways to insert $\a_{10}$ or $\a_8$ between two consecutive intersection points of $\g \cap \a_9$ to avoid a box surgery.}
\label{pattern1}
\end{figure}

\begin{lemma}
\label{lemma: exists higher index}
Let $a_{9}^1 , a_{9}^2, \cdots , a_{9}^{k} \subset \g \cap \a_{9}$ be the listing of consecutive endpoints in the order from left to right on the left end of the $k (k\geq 11)$-stack rainbow formed by the parallel arcs of $\a_{9} \setminus \g$. There must be an $\a_{10}$ in the rainbow. 
\end{lemma}

\begin{figure}[htbp]
\centering
\scalebox{.90}{\includegraphics[width=0.90 \textwidth]{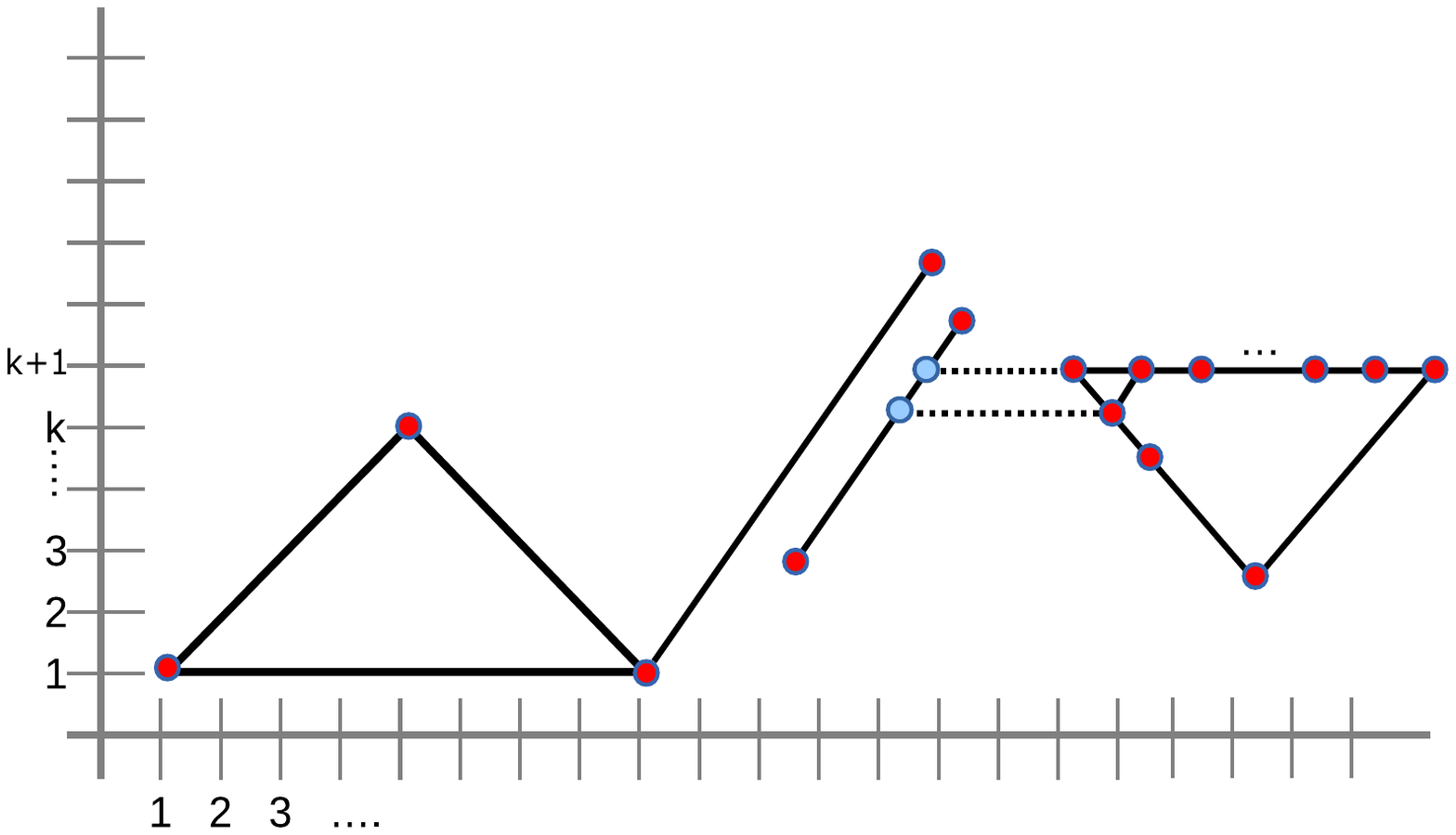}}
\caption{\small  A degenerate hexagon surgery of type 2 occurs due to the existence  of the two leftmost consecutive points of $\a_9 \cap \g$. }
\label{adjacent}
\end{figure}

\begin{proof}
 Between $ a_{9}^i$ and $a_{9}^{i+1}$ we must have either $a_{10}$ or $a_{8}$ intersect $\g$ by the efficiency of a geodesic. Otherwise, we can do a box surgery for the two consecutive intersection points, as illustrated in Figure~\ref{pattern1}. 
 
Suppose that there is no $a_{10}$ intersecting the reference arc $\g$ between $a_{9}$'s. In particular, there is no $a_{10}$ between the leftmost consecutive points of $\a_9 \cap \g$. To avoid a box surgery, there must be an $a_{8}$ instead. By induction, there is an $a_{7}$ on the $a_{8}$-level and so on. The upside down triangle in the right of Figure \ref{adjacent} illustrates this observation. 

Since the first (previous) rainbow is at bottom left of the current rainbow, there is an ascending segment blocking the two horizontal edges as the intersection sequence is in the sawtooth form. If there is no $\a_8$ in this ascending segment (i.e. it does not block the bottom edge), then it must have $\a_{10}$ to avoid a box surgery. On the other hand, if there is no $\a_9$ in this ascending segment (i.e., it does not block the top edge), then the leftmost $\a_9$ of the upside down triangle should be located on the top of the previous rainbow.  It follows that the two consecutive endpoints, $a_{9}^i$ and $a_{9}^{i+1}$, creates a degenerate hexagon surgery of type 2. See the bottom right figure in Figure \ref{BoxHexagons}. This violates the efficiency of the geodesic, and it completes the proof of the existence of $\a_{10}$. 
\end{proof}

As we mentioned before, we want to search the subarcs of $\a_j \setminus \g$ with highest index $j$ and subarcs of $\a_i \setminus \g$ with lowest index $i$ trapped in the rainbow formed by the subarcs of $a_9 \setminus \g$  with the constraint that the distance between $\a_j$ and $\a_i$ is minimal. 

\begin{lemma}
\label{lemma: formua of S(k)}
Suppose there is a $k$-stack $(k\geq 11)$ rainbow formed by the parallel arcs of $\a_{9} \setminus \g$. The minimal distance $S(k)$ between the curves of highest index and of lowest index trapped in the $k$-stack rainbow is as follows:
\begin{equation}
\label{equation}
    S(k)= 
\begin{cases}
    2p ,& \text{if } k = 3p, \ \text{for} \ p \geq 1\\
    2p + 1,              & \text{if } k = 3p + 1, \ \text{for} \ p \geq 1\\
    2p + 2,              & \text{if } k = 3p + 2, \ \text{for} \ p \geq 1
\end{cases}
\end{equation}
\end{lemma}

\begin{proof}
By Lemma \ref{lemma: exists higher index}, there exists an $\a_{10}$ in the $k$-stack rainbow.
More generally, there is an ascending segment starting with $\a_9$ between $a_{9}^i$ and $a_{9}^{i+1}$. Suppose there is no $a_{8}^{i+1}$ below $a_{9}^{i+1}$. The next ascending segment must be lower to avoid a box surgery. See Figure~\ref{pattern2}. Recall that our goal is to reduce the distance, so the highest index should be exactly one less.

To put more intersection points of $\g \cap \a_9$ on $\g$, we can insert a single $a_9^{i+2}$ next to $a_9^{i+1}$. On the left side of $a_9^i$, we need to insert $a_8^i$ to make $a_9^{i-1}$ a single dot. Hence, if there is no smaller index of $a_9^{i+1}$, the green box in Figure~\ref{pattern2} contains a local optimal pattern. 

\begin{figure}[htbp]
\labellist
\tiny \hair 2pt
\pinlabel $a_9^i$ at 393 638
\pinlabel $a_9^{i+1}$ at 470 638

\pinlabel $a_9^i$ at 393 490
\pinlabel $a_9^{i+1}$ at 470 490

\pinlabel $a_9^i$ at 393 345
\pinlabel $a_9^{i+1}$ at 470 345

\pinlabel $a_9^i$ at 393 203
\pinlabel $a_9^{i+1}$ at 465 203
\pinlabel $a_9^{i+2}$ at 510 203

\pinlabel $a_9^{i-1}$ at 353 50
\pinlabel $a_9^i$ at 405 50
\pinlabel $a_8^i$ at 375 15
\pinlabel $a_9^{i+1}$ at 465 50
\pinlabel $a_9^{i+2}$ at 505 50

\endlabellist
\centering
\scalebox{.90}{\includegraphics[width=0.90 \textwidth]{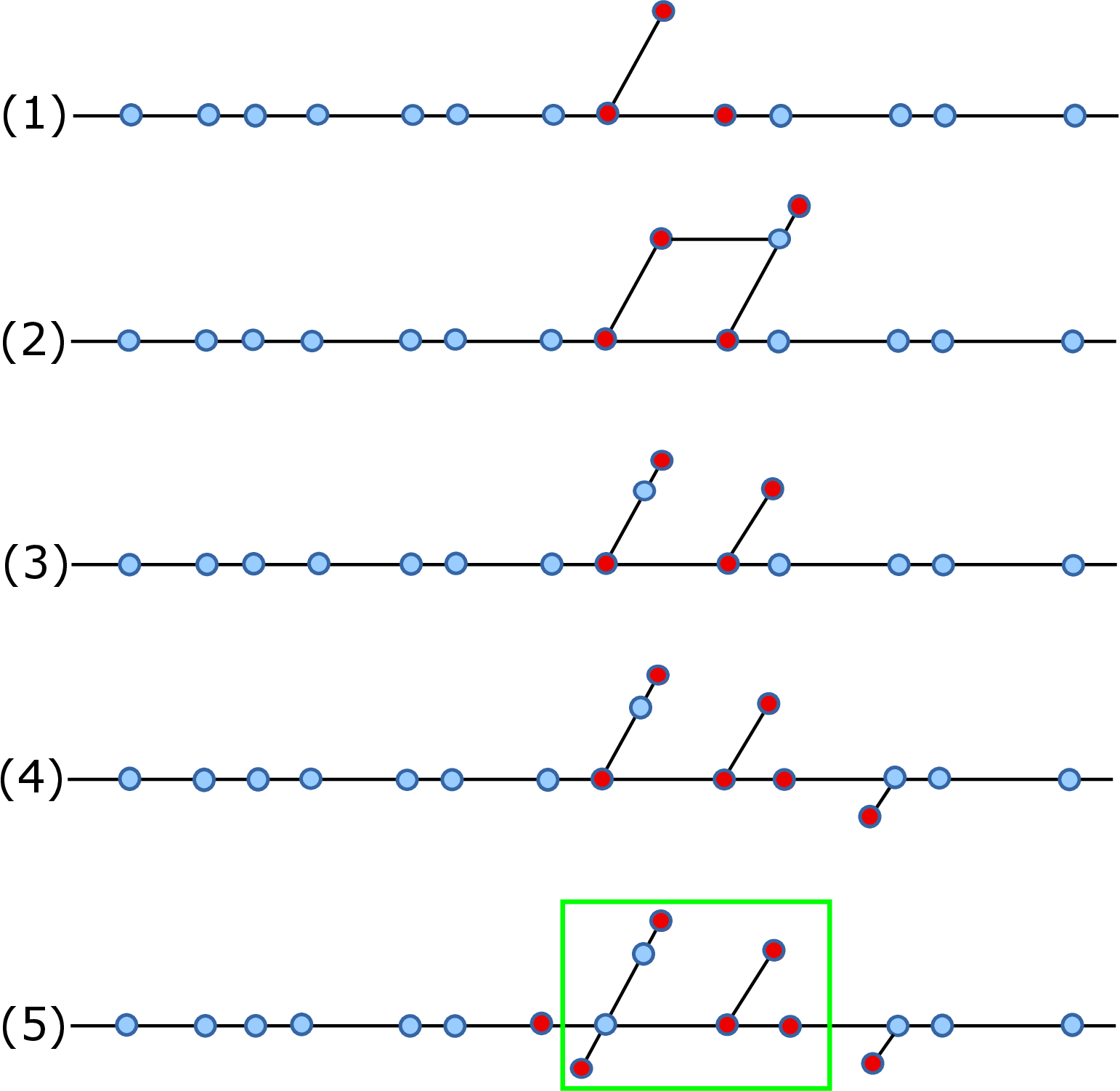}}
\caption{\small  (1) Ascending segment starts with $a_9^i$;  (2) A box surgery occurs if the next ascending segment is not lower; (3) The highest index of the ascending segment is exactly one less; (4) One more $a_9^{i+2}$ is added; (5) Insert an $a_8^i$ below $a_9^i$. }
\label{pattern2}
\end{figure}

\begin{figure}[htbp]
\labellist
\tiny \hair 2pt
\pinlabel $a_9^i$ at 395 715
\pinlabel $a_9^{i+1}$ at 490 715
\pinlabel $a_8^{i+1}$ at 470 693

\pinlabel $a_9^i$ at 395 570
\pinlabel $a_9^{i+1}$ at 475 570
\pinlabel $a_9^{i+2}$ at 525 570
\pinlabel $a_8^{i+2}$ at 490 544

\pinlabel $a_9^i$ at 395 420
\pinlabel $a_9^{i+1}$ at 475 420
\pinlabel $a_9^{i+2}$ at 525 420
\pinlabel $a_8^{i+2}$ at 490 392

\pinlabel $a_9^i$ at 395 250
\pinlabel $a_9^{i+1}$ at 475 250
\pinlabel $a_9^{i+2}$ at 525 250
\pinlabel $a_8^{i+2}$ at 490 225

\pinlabel $a_9^i$ at 400 90
\pinlabel $a_8^i$ at 375 62
\pinlabel $a_9^{i+1}$ at 473 85
\pinlabel $a_9^{i+2}$ at 522 90
\pinlabel $a_8^{i+2}$ at 485 62

\endlabellist
\centering
\scalebox{.90}{\includegraphics[width=0.90\textwidth]{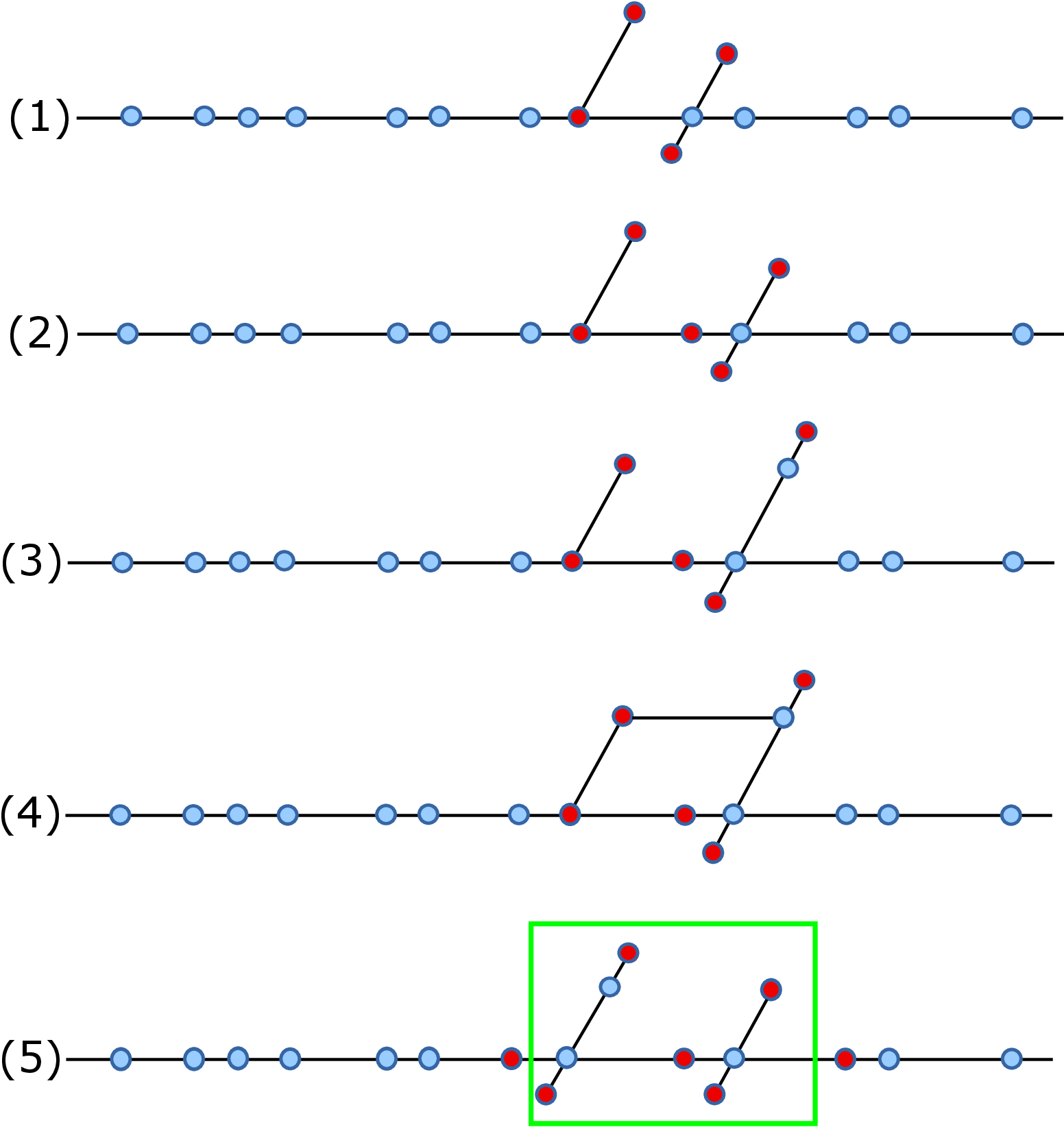}}
\caption{\small (1) An $a_8^{i+1}$ is below $a_9^{i+1}$;  (2) One more $\a_9$ is inserted; (3) Next ascending segment is not lower; (4) A hexagon surgery occurs; (5) The highest index of the ascending segment is exactly one less than that of the previous segment. }
\label{pattern3}
\end{figure}
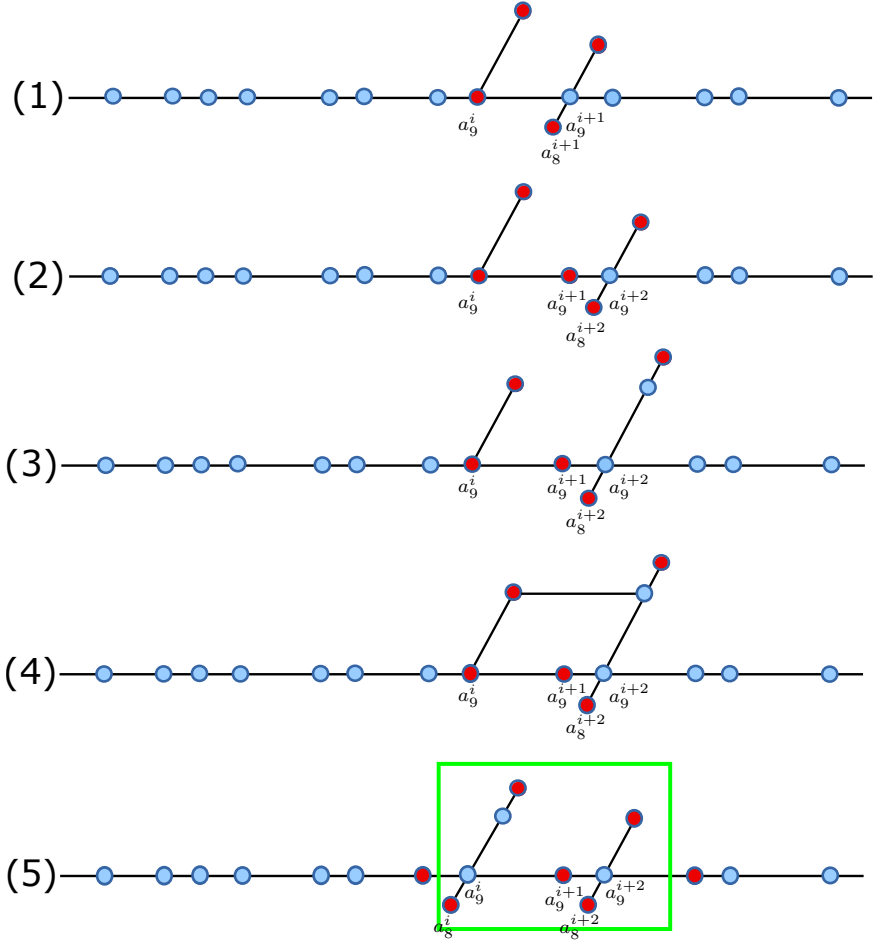

On the other hand, if $a_9^{i+1}$ does have $a_8^{i+1}$, then we can insert an additional $\a_9$ as illustrated in (2) of Figure~\ref{pattern3}. That means the additional $\a_9$ is $a_9^{i+1}$ and the original $a_9^{i+1}$ becomes $a_9^{i+2}$. The (3) and (4) in Figure~\ref{pattern3} show that the next ascending segment should be lower. Otherwise, a degenerate hexagon surgery occurs. To reduce the indices slowly, the highest index of the ascending segment should be one less. The green box in Figure~\ref{pattern3} contains another local optimal pattern.    

Both local optimal patterns have three $\a_9$'s, one of which is a single dot, and the other two are in two ascending segments. The highest index of the right one is exactly one less than that of the left one. The difference lies in the number of $\a_8$'s. The first pattern only contains a single $\a_8$, while the second one has two $\a_8$'s.

\begin{figure}[htbp]
\centering
\scalebox{.90}{\includegraphics[width=0.90\textwidth]{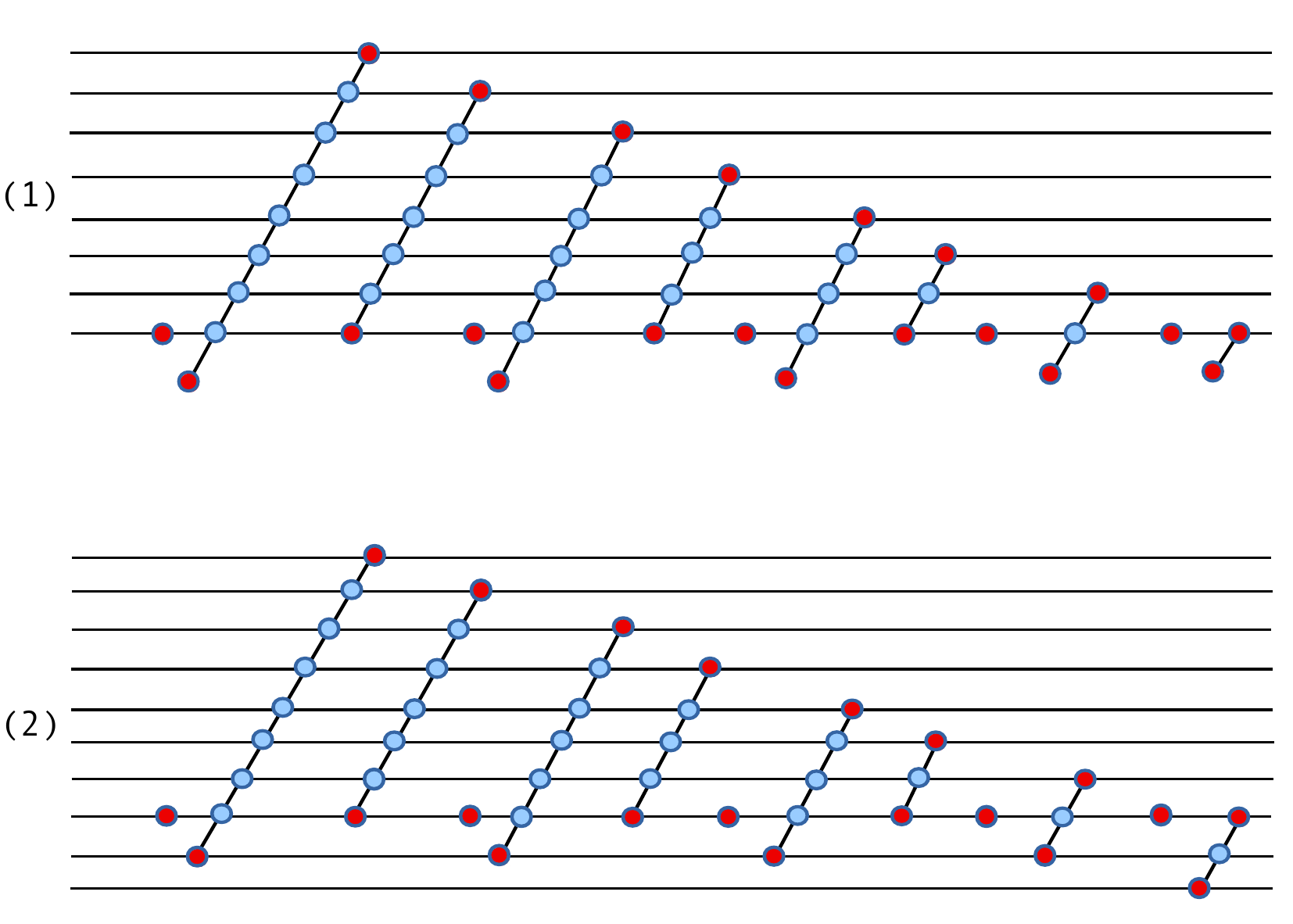}}
\caption{\small  (1) Iterate the first pattern; (2) Extend the dot graph below $\a_9$ level. }
\label{pattern4}
\end{figure}

\begin{figure}[htbp]
\centering
\scalebox{.80}{\includegraphics[width=0.90\textwidth]{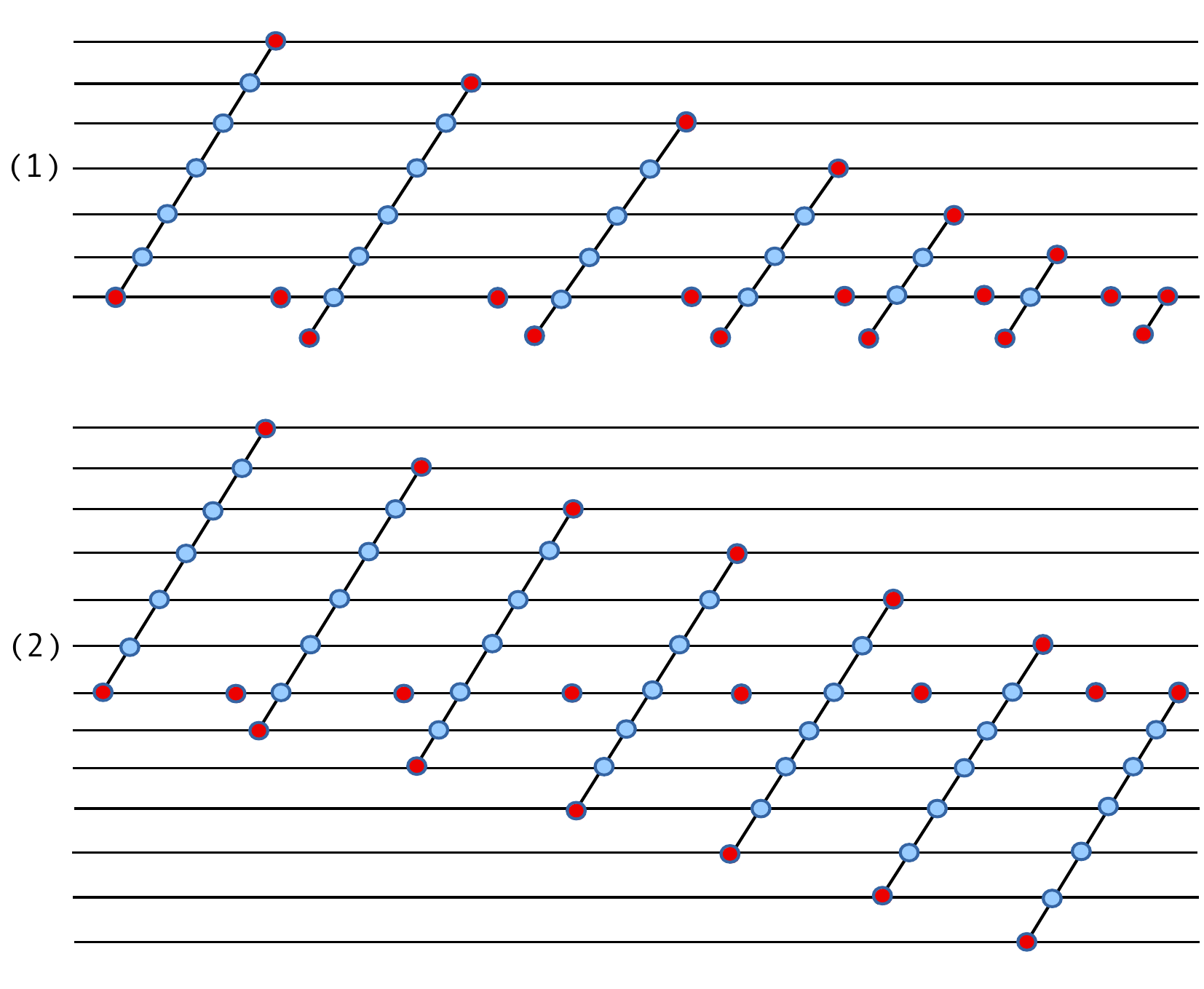}}
\caption{\small (1) Iterate the second pattern; (2) Extend the dot graph below $\a_9$ level.  }
\label{pattern5}
\end{figure}

To put it simple, if we iterate the two patterns independently, we will have the top illustrations in Figures~\ref{pattern4} \& \ref{pattern5}. The violation of hexagon surgery forces the highest index of ascending segments is larger than that on the right. To be optimal, the highest indices decrease exactly one as we move along. If we extend the graph below $\a_9$, we will observe a big difference between the two patterns. 

In the first pattern, since there is only one $\a_8$ in each pattern, it is unnecessary to add more $\a_8$ or $\a_7$ except for the rightmost ascending segment. It will violate the efficiency if we don't put an $\a_7$. The situation is quite different in the second pattern, as a hexagon surgery occurs in each pattern. To get rid of all these surgeries, we add the ascending segments as in Figure~\ref{pattern5}.  In comparison to the previous pattern, it will significantly increase the distance, so it is not the worst scenario. It follows that the worst case only consists of the first pattern. Generally, two patterns can be mixed, but each second pattern will add more curves in the rainbow. Hence, the worst scenario only contains the first pattern.

With this in mind, we define the function $S(k)$ to be the difference of the highest index $j$ and the lowest index $i$ in Figure~\ref{pattern4} type pattern that has $k$ dots at $\a_9$.  Same to the first rainbow, we can construct a new curve by the subarc $\a_i$ and $\g$, denoted as $\a_i'$. It follows that $d(\a_{j+1},\a_i')\geq j - i - 1 = S(k)-1$, because $d(\a_i',\a_i)\leq 2$. Each pattern has three dots, so we have the following three cases for $S(k)$.

\begin{figure}[htbp]
\centering
\scalebox{.90}{\includegraphics[width=0.90\textwidth]{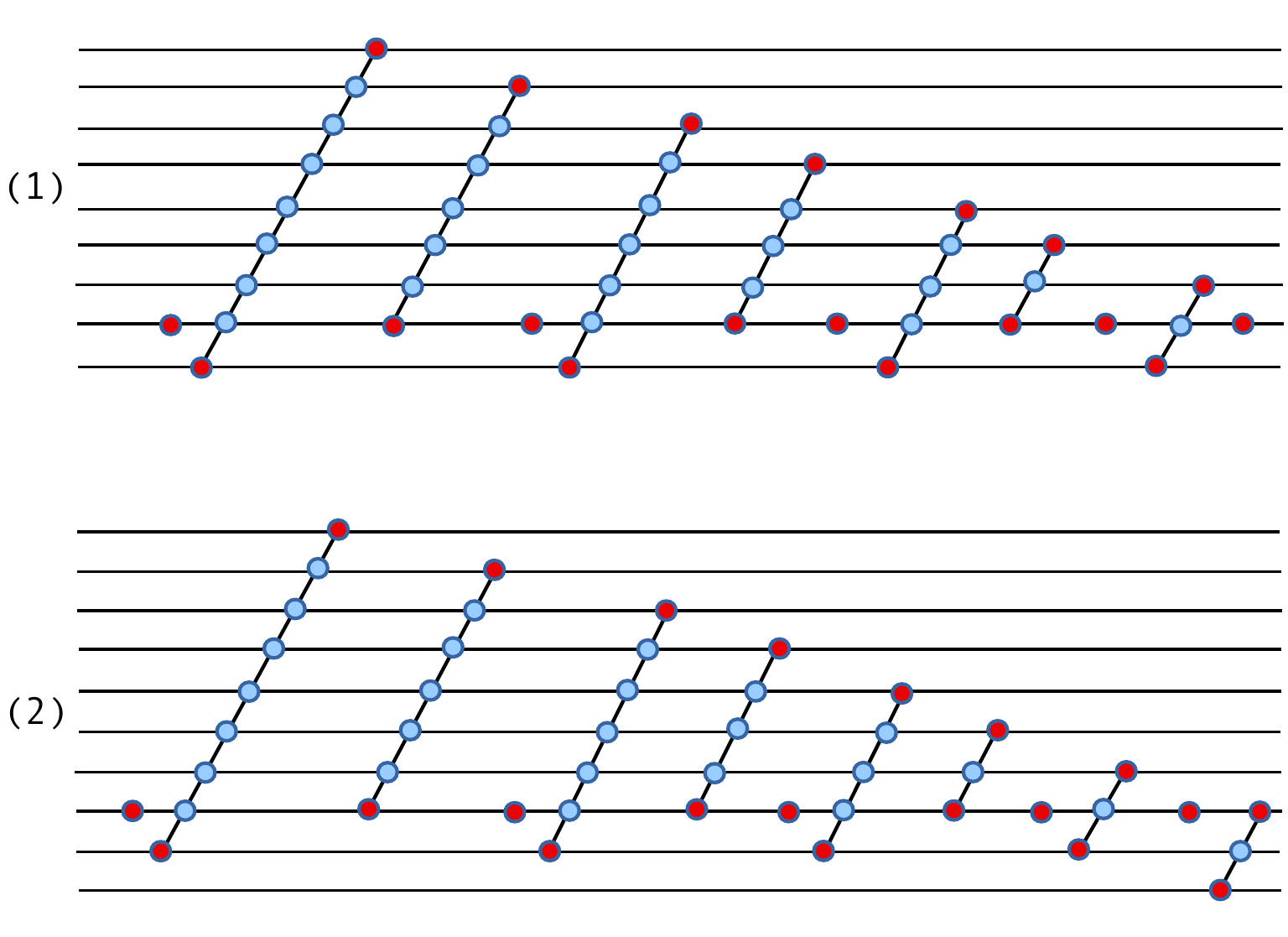}}
\caption{\small (1) The case when $k=12$, $S(12)=8$; (2) The case when $k=13$, $S(13)=9$.}
\label{pattern6}
\end{figure}

In Figure~\ref{pattern4},  the intersection number $k = 13$, and we have $S(13)=9$. When $k=12$, we can remove the rightmost ascending segment, as shown in Figure~\ref{pattern6}. In this case, $S(12)=8$. When $k=14$, either we can add a leftmost ascending segment, or add two leftmost ascending segment and remove the rightmost one. Figure~\ref{pattern7} illustrates this case, but we have $S(k)=10$ for both cases.

\begin{figure}[htbp]
\centering
\scalebox{.90}{\includegraphics[width=0.90\textwidth]{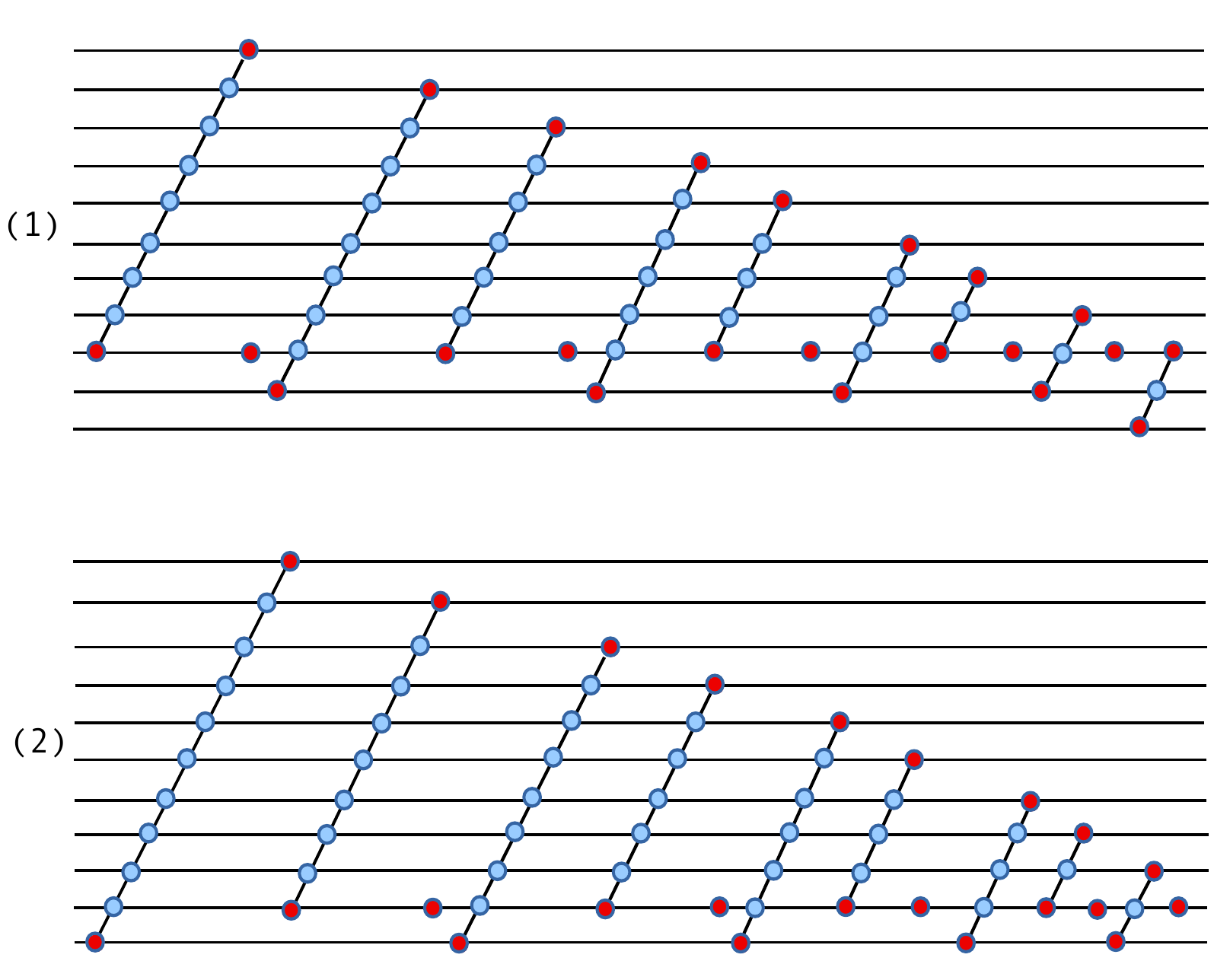}}
\caption{\small (1) and (2) are two cases when $k=14$, $S(14)=10$. }
\label{pattern7}
\end{figure}

More generally, the formula of $S(k)$ is as follows.

\begin{equation}
\label{equation}
    S(k)= 
\begin{cases}
    2p ,& \text{if } k = 3p, \ \text{for} \ p \geq 1\\
    2p + 1,              & \text{if } k = 3p + 1, \ \text{for} \ p \geq 1\\
    2p + 2,              & \text{if } k = 3p + 2, \ \text{for} \ p \geq 1
\end{cases}
\end{equation}

\end{proof}

\section{Proof of super efficiency}
\label{section: super efficiency}

Now we can repeat our previous calculations in section \ref{section: Obstruction to surgeries} to prove the main theorem.  

\begin{proof}[Proof of Theorem \ref{theorem: super efficiency}]

For the case $ g >2$, substituting $3p$  for $k$ on the right and $S(k)-1 = 2p - 1$ for $k$ on the left in equation \ref{genus-g-a1}, we get
\begin{equation}
\frac{2^{(2p - 1) -3} \cdot (2g-1)}{6g - 8} > 3p - 1 .
\end{equation}

In the case of  $g\geq 3$, the solution is $p \geq 5$, so $k = 15$.

Let us deal with the other two cases. When $k = 3p + 1$, 
\begin{equation}
\label{highest k}
\frac{2^{((2p+1) - 1) -3} \cdot (2g-1)}{6g - 8} > (3p+1) - 1 .
\end{equation}

In the case that $g = 3, 4, 5, 6, 7$, the solution is $p \geq 4$, so $k = 13$. It is $p \geq 5$ for $g \geq 8$, so $k=16$.

The other case is $k = 3p + 2$, and we have
\begin{equation}
\frac{2^{((2p+2) - 1) -3} \cdot (2g-1)}{6g - 8} > (3p+2) - 1 .
\end{equation}

The solution is $p \geq 4$, so $k = 14$ for $g\geq 3$.

Now returning to our previous equation \ref{genus-g-a1} calculation, we need to raise the value of $k$ from $8$ to $k = 16$ for $g \geq 3$.
To do this we need $| \alpha_1 \cap \g | > 15 \cdot (6g - 8 )$ .

For genus $2$, our equations to solve are

\begin{equation}
\frac{2^{(2p - 1) - 4} \cdot (12)}{4} > 3p - 1 .
\end{equation}

\begin{equation}
\frac{2^{((2p+1) - 1) - 4} \cdot (12)}{4} > (3p+1) - 1 .
\end{equation}

\begin{equation}
\frac{2^{((2p+2) - 1) - 4} \cdot (12)}{4} > (3p+2) - 1 .
\end{equation}

For the three cases, the solutions are $p \geq 4$, $p \geq 3$ and $p \geq 3$, then $k = 12$, $k=10$ and $k=11$, respectively.

Now returning back to our previous equation \ref{genus-2-a1} calculation, we need to raise the value of $k$ from $5$ to $k = 12$.
To do this we need $| \a_1 \cap \g | > 11 \cdot(4) = 44$, for $g = 2$.
Then we have our result  that the geodesic of minimal complexity is initially super  efficient with the stated bound.

\end{proof}

\begin{remark}
For genus $g \geq 7$,
it is not possible to improve the bounds of Theorem \ref{theorem: super efficiency} by carrying out the previous calculation with Bowditch's growth estimate in \S \ref{subsection: estimate}, where the estimate becomes
$$(\sqrt{2})^d \cdot (g-2) < |\a \cap \b|.$$
This observation can be illustrated by the Figure~\ref{comparison}.
\end{remark}

\begin{figure}[htbp]
\centering
\scalebox{0.60}{\includegraphics[width=0.80\textwidth]{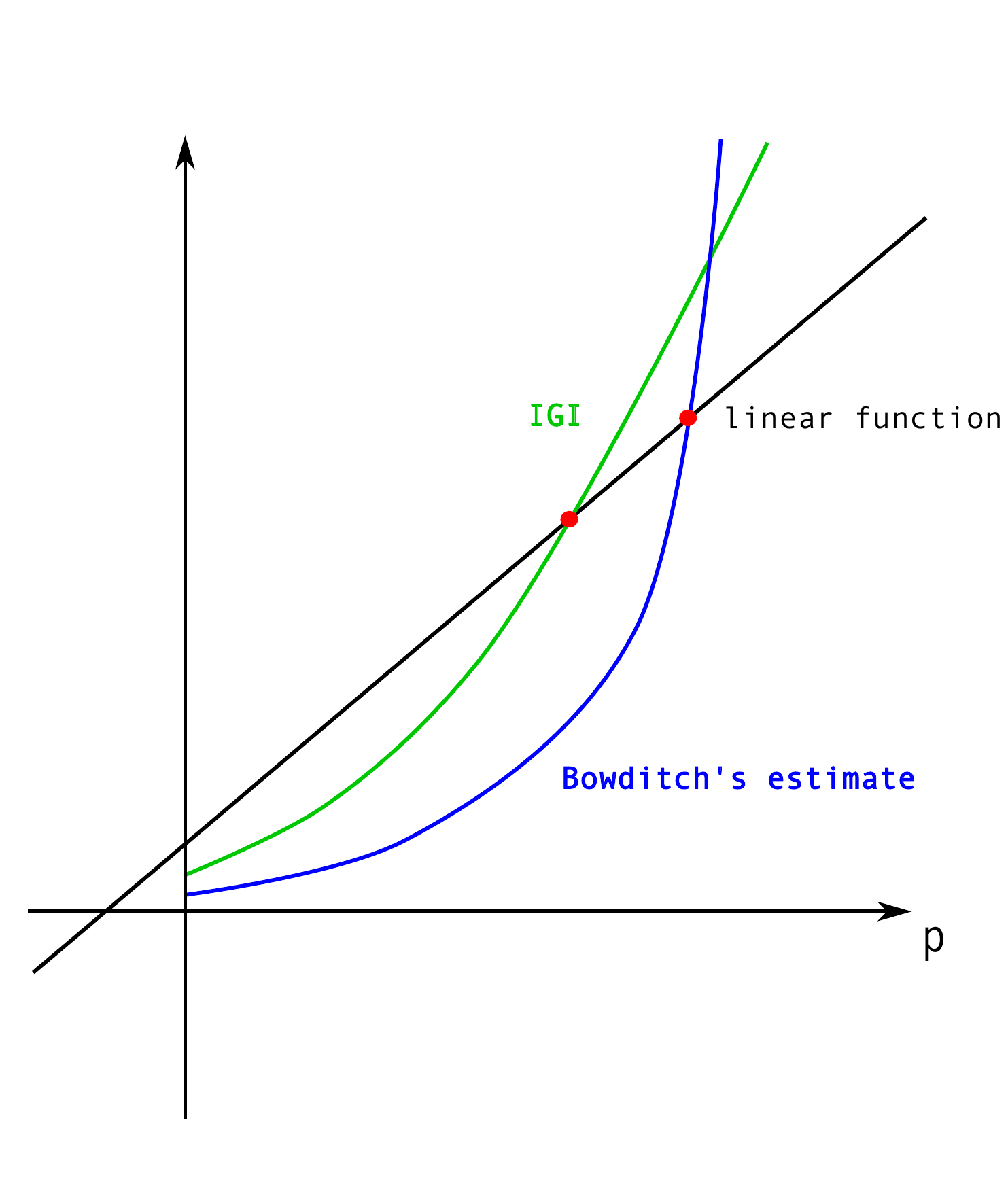}}
\caption{\small Three functions when the genus $g \geq 7$.}
\label{comparison}
\end{figure}

Once we have Theorem \ref{theorem: super efficiency}, the proof of Corollary \ref{corollary: candidates for v1} is analogous to that of Theorem 1.1 in \cite{BMM}.

\begin{proof}[Proof of Corollary \ref{corollary: candidates for v1}] 
 Let $\a_0, \a_1$ and $\a_d$ be pairwise minimal intersecting representatives of $v_0, v_1$ and $v_{\text{d}}$ without any triple points. For the polygons in $S_g \setminus (\a_0 \cup \a_d) $ that are non-hexagons, we can cut through them along some reference arcs to make them to be hexagons. With this in mind, we end up with polygons that are the rectangles and hexagons. By Euler characteristic calculation, there are $4g-4$ hexagons. Since the reference arcs in rectangles are parallel to the adjacent ones in the hexagons, then they are only counted once. It implies that there are $6g-6$ reference arcs in total. Up to homotopy, the intersection number of $\a_1$ with each reference arc determines $\a_1$. By  Theorem \ref{theorem: super efficiency}, the choice of intersection number is at most $15 \cdot (6g-8) + 1$ for each reference arc, so there are at most $[15 \cdot (6g-8) + 1]^{6g-6}$ candidates for $v_1$. For $g = 2$, the number of candidates is $45^{6}$.
 
\end{proof}

\section{subsection: algorithm}
\label{subsection: algorithm}
Continuing our discussion of \S~\ref{algorithm} concerning the implementation of a distance algorithm utilizing super efficient geodesics, the key calculation is resolving the value of $k$ for the equation \ref{genus-2-a1}.   We finally settle on $k = 16 $ as the smallest needed value when calculating the solution to equation \ref{highest k}.  Thus, we have the geometric condition that there must be fewer than $16$ parallel arcs of $\a_1 \setminus \g$ so as to prevent that growth of parallel arcs of $\a_{17} \setminus \g$, and other $\a_i {\rm 's}$.  This seems a very checkable condition for candidates of $\a_1$.  Viewed from this perspective and observing that this condition is independent of genus, the linear bounds on $|\a_1 \cap \g|$ in Theorem \ref{theorem: super efficiency} can be misleading.  Independent of distance and genus, the first step in distance algorithm based upon super efficient geodesics would be to produce a list of candidate curves for $\a_1$ with the requirement, that with respect to any reference arc, there are fewer than $16$ parallel arcs as understood in statement of Proposition \ref{proposition: calculate k}. 
 
Seen from this perspective, super efficient geodesics achieve a high degree of economy in both variables of distance and genus.  The cost of guaranteeing at least sixteen parallel arcs comes from the application of the pigeonhole principle as stated in the last sentence of the proof of Proposition \ref{proposition: calculate k}.

Having a {\em universal bound}---independent of distance and genus---on the number of parallel arcs of $\a_1 \cap \g$, for some reference arc, is striking.  At the moment there is no proof of $\delta$-hyperbolic utilizing the technology of efficient geodesics.  We know from \cite{Aougab, BB} that there is a uniform $\delta$, independent of genus.  Possessing this common attribute of universality, it is tempting to speculate that $ k = 16 $ can be leveraged to produce an efficient-geodesic-proof of hyperbolicity of the curve complex.

\bibliographystyle{plain}

\end{document}